\documentclass[]{interact}

\usepackage{epstopdf}
\usepackage[caption=false]{subfig}

\usepackage{graphicx,amssymb, verbatim,enumerate,amsmath,color}
\usepackage{amsthm}
\usepackage{upgreek}

\usepackage{mathrsfs}
\usepackage{mathabx}

\usepackage[numbers,sort&compress]{natbib}
\bibpunct[, ]{[}{]}{,}{n}{,}{,}
\makeatletter
\def\NAT@def@citea{\def\@citea{\NAT@separator}}
\makeatother

\theoremstyle{plain}
\newtheorem{theorem}{Theorem}[section]
\newtheorem{lemma}[theorem]{Lemma}
\newtheorem{corollary}[theorem]{Corollary}

\theoremstyle{definition}

\newtheorem{example}[theorem]{Example}

\theoremstyle{remark}
\newtheorem{remark}{Remark}

\begin{document}

\articletype{Article}

\title{On Approximation Properties of Generalized Lupa\c{s} Type Operators Based on
Polya Distribution with Pochhammer $k$-Symbol}

\author{
	\name{\"{O}vg\"{u} G\"{u}rel Y\i lmaz\textsuperscript{a}, Rabia Akta\c{s} \thanks{Corresponding author. Email:ovgu.gurelyilmaz@erdogan.edu.tr}\textsuperscript{b}, Fatma Ta\c{s}delen\textsuperscript{b}, and Ali Olgun \textsuperscript{c}}
	\affil{\textsuperscript{a}Recep Tayyip Erdogan University, Department of Mathematics, 53100 Rize,
		Turkey; \textsuperscript{b}Ankara University, Faculty of Science, Department of Mathematics, 06100, Tando\u{g}an, Ankara, Turkey; \textsuperscript{c}K\i r\i kkale University, Department of Mathematics, Yah\c{s}ihan
		71450, K\i r\i kkale, Turkey}
}

\maketitle

\begin{abstract}
In our present investigation, we are concerned with the Kantorovich variant of
Lupa\c{s}-Stancu operators based on Polya distribution with Pochhammer
$k$-symbol. We briefly give some basic properties of the generalized operators
and by making use of these results, we investigate convergence properties of
the studied operators. Furthermore, the rate of convergence of these operators
is obtained and Voronovskaja type theorem for the pointwise approximation is
established. Then we construct bivariate generalization of the operators and
we discuss some convergence properties. Finally, taking into account some
illustrative graphics, we conclude our study with the comparison of the rate
of convergence between our operators and other operators which are mentioned
in the paper.
\end{abstract}


\section{Introduction}
In the field of approximation theory, Bernstein operators have a considerable
importance since it is the major key for the proof of the Weierstrass
approximation theorem. In 1912, Bernstein \cite{Bernstein} presented the
well-known Bernstein operators of order $n\in%
\mathbb{N}
$ in the following form
\begin{equation}
B_{n}\left(  f;x\right)  =\sum_{m=0}^{n}\left(
\begin{array}
[c]{c}%
n\\
m
\end{array}
\right)  x^{m}\left(  1-x\right)  ^{n-m}f\left(  \frac{m}{n}\right)  ,
\label{Bernstein}%
\end{equation}
where $f\in C\left[  0,1\right]  \  \left(  \text{real valued continuous
function on }\left[  0,1\right]  \right)  $. The fact that they are
functional in studying many problems, convenient in computer-aided studies and
also they have a simple representation, important generalizations
and applications has motivated a great number of authors to study intensively up to now. We refer the
readers \cite{Acar, Acu, AMS, Cetin, MO, Opris, Ost} for some studies about
these operators.

In the year 1968, Stancu \cite{Stancu} introduced the operators $P_{n}%
^{\left \langle \alpha \right \rangle }:C\left[  0,1\right]  \rightarrow C\left[
0,1\right]  \ $with a nonnegative parameter $\alpha,$%
\begin{equation}
P_{n}^{\left \langle \alpha \right \rangle }\left(  f;x\right)  =\sum_{m=0}%
^{n}p_{n,m}^{\left \langle \alpha \right \rangle }\left(  x\right)  f\left(
\frac{m}{n}\right)  \label{Stancu}%
\end{equation}
where $p_{n,m}^{\left \langle \alpha \right \rangle }$ is defined by%
\begin{equation}
p_{n,m}^{\left \langle \alpha \right \rangle }\left(  x\right)  =\left(
\begin{array}
[c]{c}%
n\\
m
\end{array}
\right)  \frac{%
{\displaystyle \prod \limits_{\nu=0}^{m-1}}
\left(  x+\nu \alpha \right)
{\displaystyle \prod \limits_{\mu=0}^{n-m-1}}
\left(  1-x+\mu \alpha \right)  }{\left(  1+\alpha \right)  \left(
1+2\alpha \right)  ...\left(  1+\left(  n-1\right)  \alpha \right)  }
\label{Polya1}%
\end{equation}
for $n\in%
\mathbb{N}
.\ $It should be noted that when $\alpha=0,\  \left(  \text{\ref{Stancu}%
}\right)  $ obviously reduces to the classical Bernstein operators given by
$\left(  \text{\ref{Bernstein}}\right)  $.

After Stancu's paper, by taking into account the special choice $\alpha
=\frac{1}{n}$ in $\left(  \text{\ref{Stancu}}\right)  $ in 1987 Lupa\c{s} and
Lupa\c{s} \cite{Lupas} constructed the operators $P_{n}^{\left \langle \frac{1}%
{n}\right \rangle }:C\left[  0,1\right]  \rightarrow C\left[  0,1\right]  \ $as
follows
\begin{equation}
P_{n}^{\left \langle \frac{1}{n}\right \rangle }\left(  f;x\right)  =\frac
{2n!}{\left(  2n\right)  !}\sum_{m=0}^{n}\left(
\begin{array}
[c]{c}%
n\\
m
\end{array}
\right)  \left(  nx\right)  _{m}\left(  n-nx\right)  _{n-m}f\left(  \frac
{m}{n}\right)  , \label{Polya2}%
\end{equation}
where $\left(  s\right)  _{m}\ $is a rising factorial also known as the
Pochhammer symbol namely,$\ $%
\begin{equation}
\left(  s\right)  _{m}=\left \{
\begin{array}
[c]{c}%
s\left(  s+1\right)  \left(  s+2\right)  ...\left(  s+m-1\right)  \ for\ m\in%
\mathbb{N}%
\\
\  \  \  \  \  \  \  \  \  \  \  \  \  \  \  \  \  \  \  \  \  \  \ 1\  \  \  \  \  \  \  \  \  \  \  \  \  \  \  \  \  \  \  \  \ for\ m=0,~s\neq
0
\end{array}
\right.  \label{Poch}%
\end{equation}
where $s$ is a real or complex number. In 2012, Miclaus \cite{Miclaus}
reconsidered the operators $\left(  \text{\ref{Polya2}}\right)  $ and in this
work, some of the properties of the operators such as moments, the remainder
term and the monotocity properties, were recalculated with a different
technique and also asymptotic behaviour of the $\left(  \text{\ref{Polya2}%
}\right)  $ was discussed. Up to now, many operators which are based on Polya
distribution have been extensively studied. We refer the reader to the
articles \cite{Agrawal, Agrawal1, Agrawal2, Cardenas, Deo, Gupta, Kajla, Neer,
Razi-Tez} and the references therein.

In 1989, Razi \cite{Razi} defined\textbf{\ }the following Kantorovich
modification of the Bernstein-Stancu operators $P_{n}^{\left \langle
\alpha \right \rangle }\left(  f;x\right)  $ given by $\left(
\text{\ref{Stancu}}\right)  $%
\[
K_{n}^{\left(  \alpha \right)  }\left(  f;x\right)  =\left(  n+1\right)
\sum_{m=0}^{n}p_{n,m}^{\left \langle \alpha \right \rangle }\left(  x\right)
\int \limits_{\frac{m}{n+1}}^{\frac{m+1}{n+1}}f\left(  t\right)  dt.
\]
and studied some approximation properties. For $\alpha=0,$ it takes the
classical Bernstein Kantorovich operators%
\begin{equation}
K_{n}\left(  f;x\right)  =\left(  n+1\right)  \sum_{m=0}^{n}\left(
\begin{array}
[c]{c}%
n\\
m
\end{array}
\right)  x^{m}\left(  1-x\right)  ^{n-m}\int \limits_{\frac{m}{n+1}}%
^{\frac{m+1}{n+1}}f\left(  t\right)  dt. \label{*}%
\end{equation}

In 2016, for $\alpha=\frac{1}{n},$ the Kantorovich modification of Lupa\c{s}
operators based on Polya distribution $P_{n}^{\left \langle \frac{1}%
{n}\right \rangle }\left(  f;x\right)  $%
\begin{equation}
D_{n}^{\ast \left(  \frac{1}{n}\right)  }\left(  f;x\right)  =\left(
n+1\right)  \frac{2n!}{\left(  2n\right)  !}\sum_{m=0}^{n}\left(
\begin{array}
[c]{c}%
n\\
m
\end{array}
\right)  \left(  nx\right)  _{m}\left(  n-nx\right)  _{n-m}\int \limits_{\frac
{m}{n+1}}^{\frac{m+1}{n+1}}f\left(  t\right)  dt~~~,~~f\in C\left[
0,1\right]  \label{Kant1}%
\end{equation}
was studied by Agrawal et al. \cite{AIK} and, local and global approximation
properties were obtained. Furthermore, the authors introduced bivariate form
of the Kantorovich modification of Lupa\c{s} operators based on Polya
distribution defined by (\ref{Kant1}) as follows%
\begin{equation}
D_{n_{1},n_{2}}^{\ast \left(  \frac{1}{n_{1}},\frac{1}{n_{2}}\right)  }\left(
f;x,y\right)  =\left(  n_{1}+1\right)  \left(  n_{2}+1\right)  \sum_{m_{1}%
=0}^{n_{1}}\sum_{m_{2}=0}^{n_{2}}p_{n_{1},n_{2},m_{1},m_{2}}^{\left(
1/n_{1},1/n_{2}\right)  }\left(  x,y\right)  \int \limits_{\frac{m_{1}}%
{n_{1}+1}}^{\frac{m_{1}+1}{n_{1}+1}}\int \limits_{\frac{m_{2}}{n_{2}+1}}%
^{\frac{m_{2}+1}{n_{2}+1}}f\left(  t,s\right)  dtds \label{bivakant}%
\end{equation}
for $f:C\left(  J^{2}\right)  \rightarrow C\left(  J^{2}\right)  ,~J=\left[
0,1\right]  $ where
\begin{align*}
p_{n_{1},n_{2},m_{1},m_{2}}^{\left(  1/n_{1},1/n_{2}\right)  }\left(
x,y\right)   &  =\frac{2n_{1}!}{\left(  2n_{1}\right)  !}\frac{2n_{2}%
!}{\left(  2n_{2}\right)  !}\left(
\begin{array}
[c]{c}%
n_{1}\\
m_{1}%
\end{array}
\right)  \left(
\begin{array}
[c]{c}%
n_{2}\\
m_{2}%
\end{array}
\right) \\
&  \times \left(  n_{1}x\right)  _{m_{1}}\left(  n_{1}-n_{1}x\right)
_{n_{1}-m_{1}}\left(  n_{2}y\right)  _{m_{2}}\left(  n_{2}-n_{2}y\right)
_{n_{2}-m_{2}}%
\end{align*}
and they gave some rates of convergence for these operators.

In 2010, Gadjiev and Ghorbanalizadeh \cite{GG} defined a new construction of
Bernstein--Stancu type polynomials as%
\begin{equation}
B_{n,\alpha,\beta}\left(  h;y\right)  =\left(  \frac{n+\beta_{2}}{n}\right)
^{n}\sum_{m=0}^{n}\left(
\begin{array}
[c]{c}%
n\\
m
\end{array}
\right)  \left(  y-\frac{\alpha_{2}}{n+\beta_{2}}\right)  ^{m}\left(
\frac{n+\alpha_{2}}{n+\beta_{2}}-y\right)  ^{n-m}h\left(  \frac{m+\alpha_{1}%
}{n+\beta1}\right)  \label{1}%
\end{equation}
where $\frac{\alpha_{2}}{n+\beta_{2}}\leq y\leq \frac{n+\alpha_{2}}{n+\beta
_{2}},$ $\alpha_{1},\beta_{1},\alpha_{2},\beta_{2}$ are positive real number
and $0\leq \alpha_{1}\leq \alpha_{2}\leq \beta_{1}\leq \beta_{2}.$

In 2020, inspired by the operator (\ref{1}) Rahman et al. \cite{RMK}
constructed a Kantorovich type Lupa\c{s}--Stancu operators based on P\'{o}lya
distribution as follows:
\begin{equation}
S_{n,\frac{1}{n}}^{\left(  \alpha,\beta \right)  }\left(  h;y\right)  =\left(
n+\beta_{1}+1\right)  \sum_{m=0}^{n}s_{n,m}^{\left(  \alpha_{2},\beta
_{2}\right)  }\left(  y\right)  \int \limits_{\frac{m+\alpha_{1}}{n+\beta
_{1}+1}}^{\frac{m+\alpha_{1}+1}{n+\beta_{1}+1}}h\left(  u\right)  du
\label{Kant2}%
\end{equation}
where
\begin{align*}
s_{n,m}^{\left(  \alpha_{2},\beta_{2}\right)  }\left(  y\right)   &
=\frac{\left(
\begin{array}
[c]{c}%
n\\
m
\end{array}
\right)  \left(  y-\frac{\alpha_{2}}{n+\beta_{2}}\right)  _{m,\frac{1}{n}%
}\left(  \frac{n+\alpha_{2}}{n+\beta_{2}}-y\right)  _{n-m,\frac{1}{n}}%
}{\left(  \frac{n}{n+\beta_{2}}\right)  _{n,\frac{1}{n}}},\\
\left(  y\right)  _{m,\frac{1}{n}}  &  =y\left(  y+\frac{1}{n}\right)  \left(
y+\frac{2}{n}\right)  ...\left(  y+\frac{m-1}{n}\right)
\end{align*}
and $y\in \left[  \frac{\alpha_{2}}{n+\beta_{2}},\frac{n+\alpha_{2}}%
{n+\beta_{2}}\right]  $ and $\alpha_{1},\beta_{1},\alpha_{2},\beta_{2}$ are
positive real number, and $0\leq \alpha_{1}\leq \alpha_{2}\leq \beta_{1}\leq
\beta_{2}.$ It is clear that for $\alpha_{1}=\beta_{1}=\alpha_{2}=\beta_{2}=0$
it reduces to operator $\left(  \text{\ref{Kant1}}\right)  .$ For $\alpha
_{2}=\beta_{2}=0,$ it gives the Kantorovich-Stancu generalization of the
operators $D_{n}^{\ast \left(  \frac{1}{n}\right)  }.~$Additionally, Rahman et
al. \cite{RMK} defined a bivariate generalization of the Kantorovich type
Lupa\c{s}--Stancu operators given by (\ref{Kant2}) as follows%
\begin{align}
S_{n_{1},n_{2},\frac{1}{n_{1}},\frac{1}{n_{2}}}^{\left(  \alpha,\beta \right)
}\left(  h;y,z\right)   &  =\left(  n_{1}+\beta_{1}+1\right)  \left(
n_{2}+\beta_{1}+1\right)  \sum_{m_{1}=0}^{n_{1}}\sum_{m_{2}=0}^{n_{2}}%
s_{n_{1},n_{2},m_{1},m_{2}}^{\left(  \alpha_{2},\beta_{2}\right)  }\left(
y,z\right) \nonumber \\
&  \times \int \limits_{\frac{m_{1}+\alpha_{1}}{n_{1}+\beta_{1}+1}}^{\frac
{m_{1}+\alpha_{1}+1}{n_{1}+\beta_{1}+1}}\int \limits_{\frac{m_{2}+\alpha_{1}%
}{n_{2}+\beta_{1}+1}}^{\frac{m_{2}+\alpha_{1}+1}{n_{2}+\beta_{1}+1}}h\left(
u,s\right)  du~ds. \label{biva2}%
\end{align}

In 2007, the notion of Pochhammer $k$-symbol was first proposed by Diaz and
Pariguan in the work \cite{Diaz}. For $\lambda \in%
\mathbb{C}
,\ $it is defined by\
\begin{equation}
\left(  \lambda \right)  _{m,k}=\left \{
\begin{array}
[c]{c}%
\lambda \left(  \lambda+k\right)  \left(  \lambda+2k\right)  ...\left(
\lambda+\left(  m-1\right)  k\right)  \ ;\  \  \ m\geq
1\  \  \  \  \  \  \  \  \  \  \  \  \  \  \  \  \  \  \\
\  \  \  \  \  \  \  \  \  \  \  \  \  \  \  \  \  \  \ 1\  \  \  \  \  \  \  \  \  \  \  \  \  \  \ ;\  \ m=0,\  \lambda
\neq0
\end{array}
\right.  \label{Pock}%
\end{equation}
where it is assumed that $m\in%
\mathbb{N}
,\ k$ nonnegative real number$.\ $It is easy to say that by taking $k=1,$ the
definition of Pochhammer $k$-symbol coincides with the usual Pochhammer symbol
which is given by $\left(  \text{\ref{Poch}}\right)  .$This investigation has
revealed many new generalizations with it, such as $k$-Gamma function,
$k$-Beta function, $k$-Zeta function, $k$ generalization of hypergeometric
function and so on. For more detail see \cite{Diaz, Li, Kokolo, Krasniqi,
Mubeen, Mubeen1}.

Our present study is motivated essentially by the theory of Polya distribution
of the operators, considered by Lupa\c{s} and Lupa\c{s} \cite{Lupas}, presented in
\cite{Miclaus} and the Pochhammer $k$-symbol given by \cite{Diaz}. Here, we
define some slight modifications of Polya distribution for the special case
$\alpha=\frac{k}{n},~k$ nonnegative real number$,$ applying the notion of
Pochhammer $k$-symbol in the definition of Polya distribution.\textbf{\ }Our
main objective of this article is to investigate such a generalization how
affects the rate of convergence of the operators. The structure of the paper
reads as follows. In section 2, we first consider Lupa\c{s} type
generalization $P_{n,k}^{\left \langle \frac{k}{n}\right \rangle }\left(
f;x\right)  $ for the case $\alpha=\frac{k}{n},~k$ nonnegative real number, of
the Bernstein-Stancu operators $P_{n}^{\left \langle \alpha \right \rangle
}\left(  f;x\right)  .$ We compare the operators $P_{n,k}^{\left \langle
\frac{k}{n}\right \rangle }\left(  f;x\right)  $ and the operators
$P_{n}^{\left \langle \frac{1}{n}\right \rangle }\left(  f;x\right)  $ with some
graphics. We show that for $0\leq k<1$\ the operators $P_{n,k}^{\left \langle
\frac{k}{n}\right \rangle }\left(  f;x\right)  $ give a better approximation
than $P_{n,k}^{\left \langle \frac{1}{n}\right \rangle }\left(  f;x\right)  $.
Then we introduce Kantorovich-Stancu modification of the Lupa\c{s}
type operator $P_{n,k}^{\left \langle \frac{k}{n}\right \rangle }$ with
Pochhammer $k$-symbol and present some fundamental results such as moments,
central moments. Convergence properties of the new operators are examined.
More precisely, we give the theorem about the uniform convergence, estimate
the rate of the convergence by means of the classical modulus of continuity
and discuss pointwise approximation via Voronovskaja type theorem.
Furthermore, in section 3 we also consider a bivariate generalization of
Kantorovich-Stancu modification of the Lupa\c{s} type operator $P_{n,k}%
^{\left \langle \frac{k}{n}\right \rangle }$.We give some preliminary results
for the bivariate\textbf{ }Lupa\c{s}-Kantorovich-Stancu type operators and
discuss some approximation properties. Finally, taking into account some
illustrative graphics, we conclude our study with the comparison of the rate
of convergence between our operators and other operators which are mentioned
in the paper for the values of the parameter $k$.

\section{Lupa\c{s} type operators by means of Pochhammer $k$-symbol}
Before we start this section, we briefly describe the Lupa\c{s} type operators
that play an important role in this paper. Let $C\left[  0,1\right]  $ be the
space of all real valued continuous functions on $\left[  0,1\right]
\ $endowed with the norm%
\[
\left \Vert f\right \Vert _{C\left[  0,1\right]  }=\sup_{x\in \left[  0,1\right]
}\left \vert f\left(  x\right)  \right \vert .
\]
Let $n\in%
\mathbb{N}
$ and $k$ nonnegative real number$.$ By taking into account the special choice
$\alpha=\frac{k}{n}$ in $\left(  \text{\ref{Stancu}}\right)  $ in view of the
notion of Pochhammer $k$-symbol, the operators $P_{n,k}^{\left \langle \frac
{k}{n}\right \rangle }:C\left[  0,1\right]  \rightarrow C\left[  0,1\right]  $
is defined by%

\begin{equation}
P_{n,k}^{\left \langle \frac{k}{n}\right \rangle }\left(  f;x\right)  =\frac
{1}{\left(  n\right)  _{n,k}}\sum_{m=0}^{n}\left(
\begin{array}
[c]{c}%
n\\
m
\end{array}
\right)  \left(  nx\right)  _{m,k}\left(  n-nx\right)  _{n-m,k}f\left(
\frac{m}{n}\right)  , \label{PolBer}%
\end{equation}
where $\left(  \lambda \right)  _{m,k}$ is a Pochhammer $k$-symbol given by
$\left(  \text{\ref{Pock}}\right)  .$ The case $k=0$ turns to the classical
Bernstein operators. For $k=1,$ it gives Lupa\c{s} operators given by
(\ref{Polya2}).

Now, we begin our study by giving moments, central moments and rate of
convergence for the operators\textbf{\ }$\left(  \text{\ref{PolBer}}\right)
.$ We first give Lemma \ref{test} , Lemma \ref{k} and Corollary \ref{j} without proof, which follows
from the results given in the paper \cite{Miclaus} for $\alpha=\frac{k}{n}$,
$k\geq0$.

Throughout this paper, $%
\mathbb{N}
$ denotes the set of positive integers and $%
\mathbb{N}
_{0}=%
\mathbb{N}
\cup \left \{  0\right \}  $ and let us denote the monomials $e_{j}\left(
t\right)  =t^{j}$ for $j\in%
\mathbb{N}
_{0}.$

\begin{lemma}
\label{test}Let $n\in%
\mathbb{N}
$ and$\ k$ be nonnegative real number$.$ Then for the operators $P_{n,k}%
^{\left \langle \frac{k}{n}\right \rangle }$ defined by $\left(
\text{\ref{PolBer}}\right)  $, we have%
\begin{align}
P_{n,k}^{\left \langle \frac{k}{n}\right \rangle }\left(  e_{0};x\right)   &
=1,\label{test1}\\
P_{n,k}^{\left \langle \frac{k}{n}\right \rangle }\left(  e_{1};x\right)   &
=x,\label{test2}\\
P_{n,k}^{\left \langle \frac{k}{n}\right \rangle }\left(  e_{2};x\right)   &
=x^{2}+\frac{\left(  k+1\right)  x\left(  1-x\right)  }{n+k},\label{test3}\\
P_{n,k}^{\left \langle \frac{k}{n}\right \rangle }\left(  e_{3};x\right)   &
=x^{3}+\frac{\left(  3n+2k-2\right)  \left(  k+1\right)  x^{2}\left(
1-x\right)  }{\left(  n+k\right)  \left(  n+2k\right)  }\nonumber \\
&  +\frac{\left(  2k+1\right)  \left(  k+1\right)  x\left(  1-x\right)
}{\left(  n+k\right)  \left(  n+2k\right)  },\label{test4}
\end{align}
\begin{align}
P_{n,k}^{\left \langle \frac{k}{n}\right \rangle }\left(  e_{4};x\right)   &
=x^{4}+\frac{\left(  k+1\right)  \left(  \left(  11n-6\right)  \left(
k-1\right)  +6\left(  n^{2}+k^{2}\right)  \right)  x^{3}\left(  1-x\right)
}{\left(  n+k\right)  \left(  n+2k\right)  \left(  n+3k\right)  }\nonumber \\
&  +\frac{\left(  k+1\right)  \left(  7n+11nk+6\left(  k^{2}-k-1\right)
\right)  x^{2}\left(  1-x\right)  }{\left(  n+k\right)  \left(  n+2k\right)
\left(  n+3k\right)  }\nonumber \\
&  +\frac{\left(  k+1\right)  \left(  n-k+6nk\left(  k+1\right)  \right)
x\left(  1-x\right)  }{n\left(  n+k\right)  \left(  n+2k\right)  \left(
n+3k\right)  }. \label{test5}%
\end{align}

\end{lemma}

\begin{corollary}
\label{j} Let $n\in%
\mathbb{N}
$ and $k$ be nonnegative real number. Then the central moments of the
operators $P_{n,k}^{\left \langle \frac{k}{n}\right \rangle }$ are given by%
\begin{align}
P_{n,k}^{\left \langle \frac{k}{n}\right \rangle }\left(  \left(  e_{1}%
-x\right)  ^{2};x\right)   &  =\frac{\left(  k+1\right)  x\left(  1-x\right)
}{n+k},\label{cent1}\\
P_{n,k}^{\left \langle \frac{k}{n}\right \rangle }\left(  \left(  e_{1}%
-x\right)  ^{3};x\right)   &  =\frac{\left(  k+1\right)  \left(  2k+1\right)
x\left(  1-x\right)  \left(  1-2x\right)  }{\left(  n+k\right)  \left(
n+2k\right)  },\nonumber
\end{align}%
\begin{align}
P_{n,k}^{\left \langle \frac{k}{n}\right \rangle }\left(  \left(  e_{1}%
-x\right)  ^{4};x\right)   &  =\frac{\left(  k+1\right)  3n\left(
-2+n+k\left(  -6-6k+n\right)  \right)  \left(  x\left(  1-x\right)  \right)
^{2}}{n\left(  n+k\right)  \left(  n+2k\right)  \left(  n+3k\right)
}\nonumber \\
&  +\frac{\left(  k+1\right)  \left(  n+k\left(  -1+6\left(  k+1\right)
n\right)  \right)  x\left(  1-x\right)  }{n\left(  n+k\right)  \left(
n+2k\right)  \left(  n+3k\right)  }. \label{cent3}%
\end{align}

\end{corollary}

\begin{lemma}
 \label{k} Let $n\in%
\mathbb{N}
$ and$\ k$ be nonnegative real number$.\ $Then for every $f\in C\left[
0,1\right]  ,$ we have%
\[
\lim_{n\rightarrow \infty}P_{n,k}^{\left \langle \frac{k}{n}\right \rangle
}\left(  f;x\right)  =f\left(  x\right)
\]
uniformly in $\left[  0,1\right]  .$
\end{lemma}

For $\alpha=\frac{k}{n}$, $k\geq0,$ we can give the next theorem from the
results given in \cite{Razi-Tez}.

\begin{theorem}
Let $n\in%
\mathbb{N}
$ and $k$ be nonnegative real number$.\ $Then for every $f\in C\left[
0,1\right]  ,$ the following inequality holds%
\begin{equation}
\left \Vert P_{n,k}^{\left \langle \frac{k}{n}\right \rangle }\left(  f;x\right)
-f\left(  x\right)  \right \Vert _{C[0,1]}\leq \frac{3}{2}\omega \left(
f,\sqrt{\frac{k+1}{n+k}}\right)  \label{surmod}%
\end{equation}
for $k=1,$ which concludes that%
\[
\left \Vert P_{n}^{\left \langle \frac{1}{n}\right \rangle }\left(  f;x\right)
-f\left(  x\right)  \right \Vert _{C[0,1]}\leq \frac{3}{2}\omega \left(
f,\sqrt{\frac{2}{n+1}}\right)
\]
where $\omega \left(  f,.\right)  $ denotes modulus of continuity of $f$
defined by
\[
\omega \left(  f,\delta \right)  =\sup_{\substack{x,t\in \left[  a,b\right]
\\ \left \vert t-x\right \vert \leq \delta}}\left \vert f\left(  t\right)
-f\left(  x\right)  \right \vert
\]
for $\delta>0.$
\end{theorem}

We compare the rate of convergence between the operators\textbf{\ }%
$P_{n,k}^{\left \langle \frac{k}{n}\right \rangle }$ in $\left(
\text{\ref{PolBer}}\right)  $ and Lupa\c{s} operators $P_{n}^{\left \langle
\frac{1}{n}\right \rangle }\left(  f;x\right)  $ with illustrative graphics. It
is observed that the operators\textbf{\ }$P_{n,k}^{\left \langle \frac{k}%
{n}\right \rangle }$ give a better approximation than $P_{n}^{\left \langle
\frac{1}{n}\right \rangle }\left(  f;x\right)  $ for $0\leq k<1.$

\begin{remark}
Let $n\in%
\mathbb{N}
$,$\ k\geq0.\ $Then for every $f\in C\left[  0,1\right]  ,$ the inequality
$P_{n,k}^{\left \langle \frac{k}{n}\right \rangle }\left(  \left(
e_{1}-x\right)  ^{2};x\right)  \leq$ $P_{n}^{\left \langle \frac{1}%
{n}\right \rangle }\left(  \left(  e_{1}-x\right)  ^{2};x\right)  $ holds such
that $k\leq1.\ $It follows that the operators $P_{n,k}^{\left \langle \frac
{k}{n}\right \rangle }$ provide a better approximation than the classical
operators $P_{n}^{\left \langle \frac{1}{n}\right \rangle }.$
\end{remark}

Now, we demonstrate the behaviour of the approximation for the operators
$P_{n,k}^{\left \langle \frac{k}{n}\right \rangle }\ $by graphical
examples.

\begin{example}
Let $f\left(  x\right)  =20x^{6}+3x^{3}-5x^{2}+2x,\ n=10$ and $k=0.1.\ $In
Figure 1, we analyse the convergence of the new operators $P_{n,k}%
^{\left \langle \frac{k}{n}\right \rangle }$, the classical operators
$P_{n}^{\left \langle \frac{1}{n}\right \rangle }$ and the classical Bernstein operators $B_{n}$ to the function $f.\ $It is seen that for $k=0.1$, $P_{n,k}^{\left \langle \frac{k}{n}\right \rangle }$ provides a better approximation than the operators $P_{n}^{\left \langle \frac{1}{n}\right \rangle }$ to the function $f$.
\end{example}

\begin{example}
Let us consider $f\left(  x\right)  =\sin \left(  6\pi x\right)  +5\sin \left(
\frac{1}{3}\pi x\right)  .$ Figure 2 illustrates the approximation process of
the operators $P_{n,k}^{\left \langle \frac{k}{n}\right \rangle }$ for $k=0.5$
fixed and the special choices of $n=10,50\ $and $100$. It can be observed that
as the value of $n$ increases,\ the approximation of the operators
$P_{n,k}^{\left \langle \frac{k}{n}\right \rangle }\ $is getting better.
\end{example}

\begin{example}
Let $f\left(  x\right)  =\sin \left(  2\pi x\right)  +2\sin \left(  \frac{1}%
{2}\pi x\right)  .$ Figure 3 presents the convergence of $P_{n,k}%
^{\left \langle \frac{k}{n}\right \rangle }$ to the function $f$ for
$n=10\ $fixed and $k=0.1,\ 0.3,\ 0.6,\ 1$ and $3.$ From this figure, it
follows that when $k$ gets smaller towards to zero, approximation is better
than others.
\end{example}

\begin{figure}[pth]
	\includegraphics*[width=7cm]{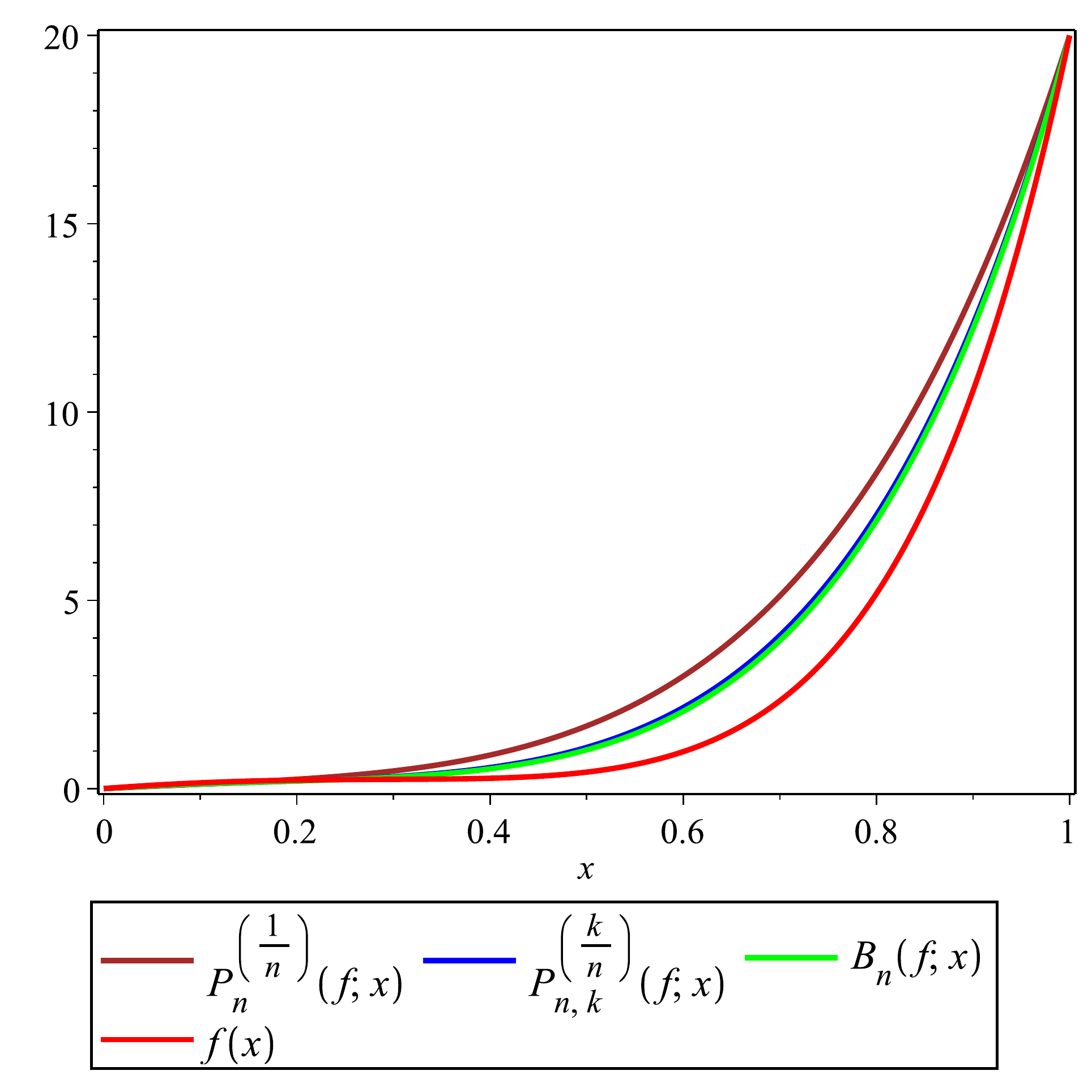}\caption{Convergence of
		$P_{n,k}^{\left \langle \frac{k}{n}\right \rangle }$, $P_{n}^{\left \langle
			\frac{1}{n}\right \rangle }$ and $B_{n}$  to the function $f$ for $n=10$ and $k=0.1$ }%
\end{figure}

\begin{figure}[pth]
\includegraphics*[width=7cm]{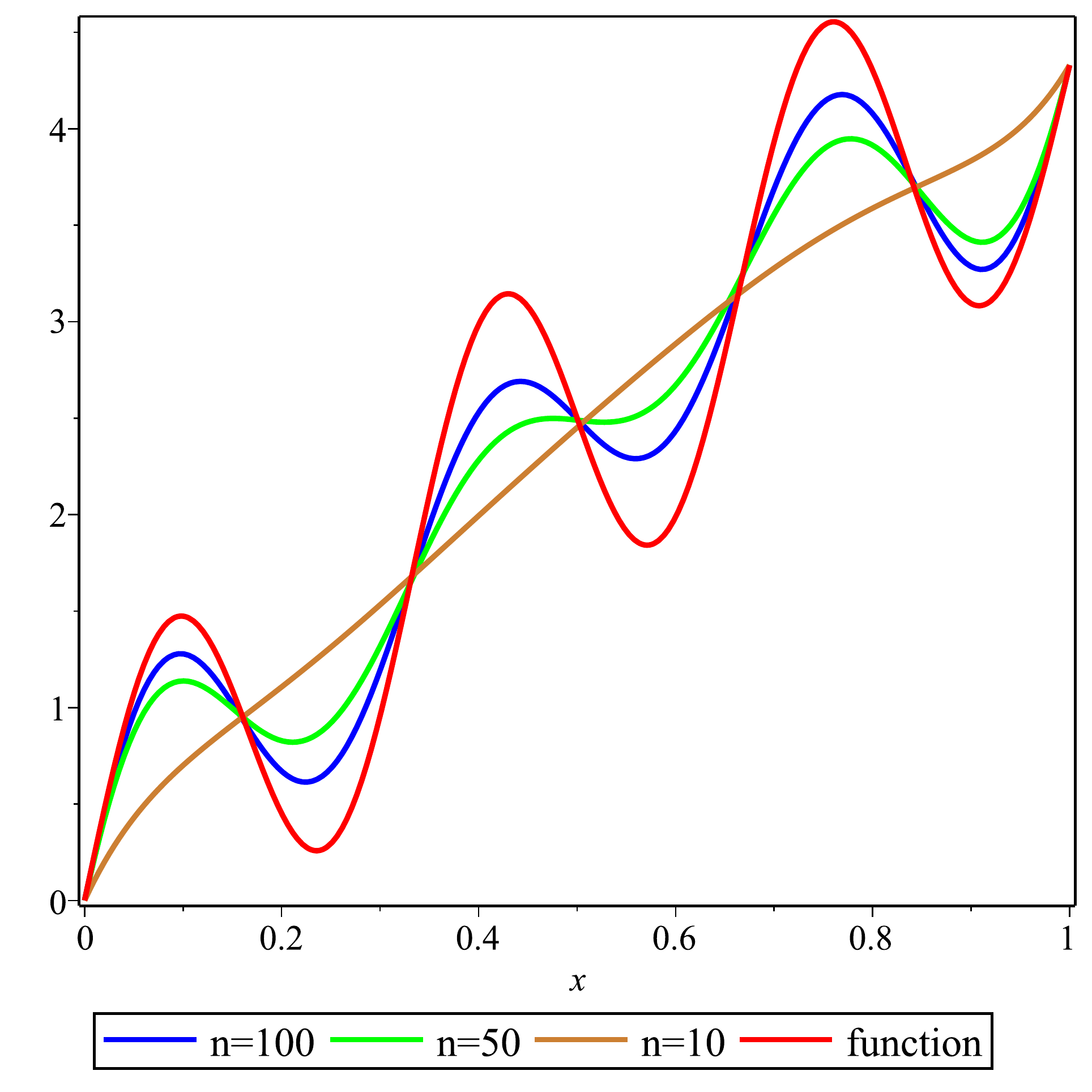}\caption{Approximation of the
operators $P_{n,k} ^{\left \langle \frac{k}{n}\right \rangle }$ for $k=0.5$
fixed and $n=10,50,100$ }%
\end{figure}

\begin{figure}[pth]
\includegraphics*[width=7cm]{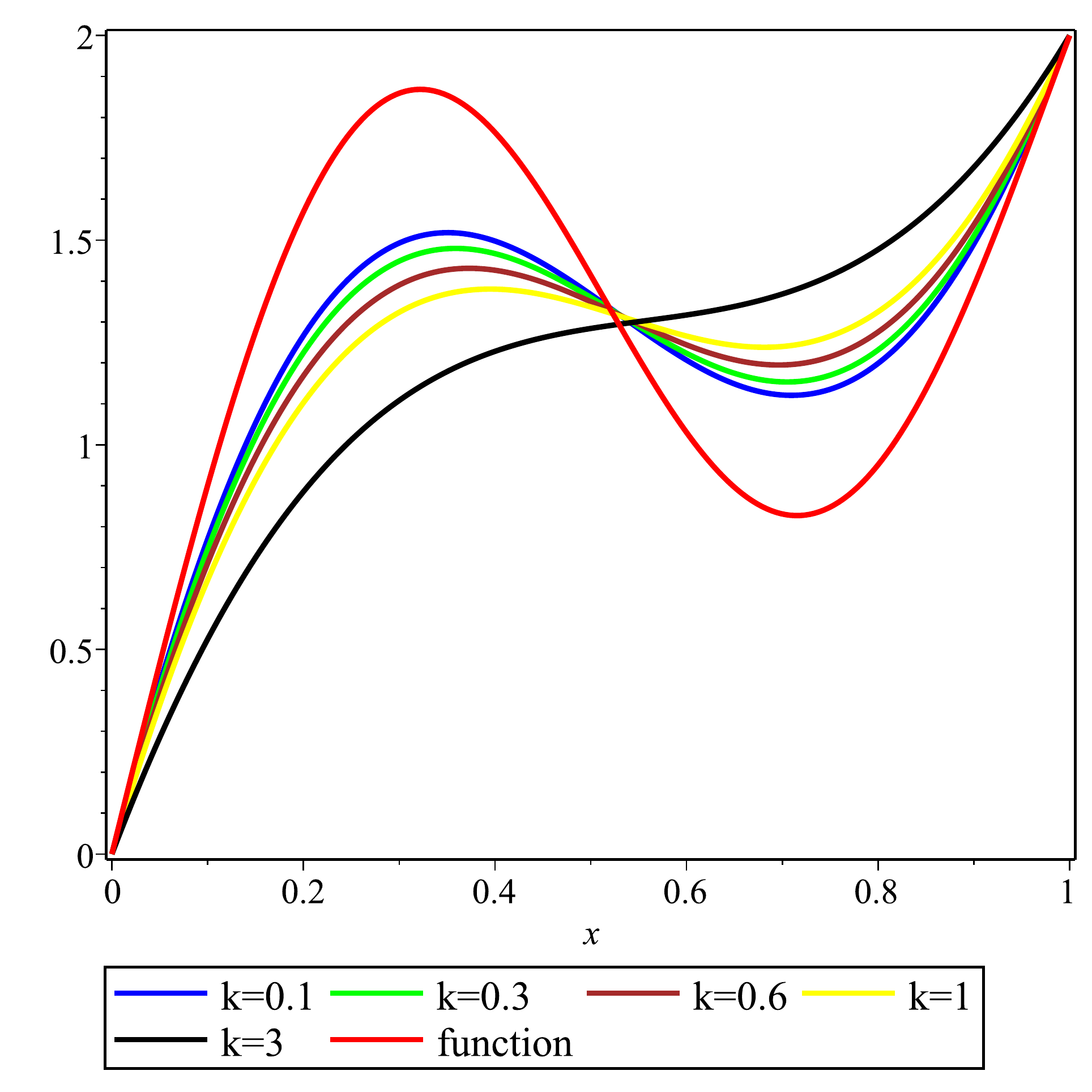}\caption{Approximation of the
operators $P_{n,k}^{\left \langle \frac{k}{n}\right \rangle }$ for $n=10$ fixed
and $k=0.1,\ 0.3,\ 0.6,\ 1,\ 3$ }%
\end{figure}

\newpage

 Now we consider a Kantorovich-Stancu modification of Lupa\c{s} type
operators\textbf{\ }$\left(  \text{\ref{PolBer}}\right)  $ as follows%
\begin{equation}
K_{n}^{\left(  \alpha,\beta,k\right)  }\left(  f;x\right)  =\frac{n+\beta
+1}{\left(  n\right)  _{n,k}}\sum_{m=0}^{n}\left(
\begin{array}
[c]{c}%
n\\
m
\end{array}
\right)  \left(  nx\right)  _{m,k}\left(  n-nx\right)  _{n-m,k}\int
\limits_{\frac{m+\alpha}{n+\beta+1}}^{\frac{m+\alpha+1}{n+\beta+1}}f\left(
t\right)  dt, \label{PolKant}%
\end{equation}
where $\alpha,\beta,k$ are nonnegative real number, and $0\leq \alpha \leq
\beta.$ In the case of $k=0,$ it reduces to the Kantorovich-Stancu
modification of Bernstein operators \cite{Barbosu}. For $k=1$, it gives the
special case $\alpha_{2}=\beta_{2}=0$ of Kantorovich type Lupas--Stancu
operators given by $\left(  \text{\ref{Kant2}}\right)  .$ For $k=1,\alpha
=\beta=0$, this operator reduces to the Kantorovich modification of the
operators $P_{n}^{\left \langle \frac{1}{n}\right \rangle }\left(  f;x\right)  $
given by $\left(  \text{\ref{Kant1}}\right)  .$

Taking into account the results given for the operators $P_{n,k}^{\left \langle
\frac{k}{n}\right \rangle }~$in Lemma \ref{test}, we
can give the next lemma.

\begin{lemma}
\label{testkant} Let $n\in%
\mathbb{N}
$,$\ k\geq0.$ Then for the operators $K_{n}^{\left(  \alpha,\beta,k\right)  }$
defined by $\left(  \text{\ref{PolKant}}\right)  $, we have
\end{lemma}

\begin{align}
K_{n}^{\left(  \alpha,\beta,k\right)  }\left(  e_{0};x\right)   &
=1,\label{testkant1}\\
K_{n}^{\left(  \alpha,\beta,k\right)  }\left(  e_{1};x\right)   &  =\frac
{nx}{n+\beta+1}+\frac{2\alpha+1}{2\left(  n+\beta+1\right)  }%
,\label{testkant2}\\
K_{n}^{\left(  \alpha,\beta,k\right)  }\left(  e_{2};x\right)   &
=\frac{n^{2}\left(  n-1\right)  x^{2}}{\left(  n+\beta+1\right)  ^{2}%
(n+k)}+\frac{\left(  \left(  2\alpha+2+k\right)  n+\left(  2\alpha+1\right)
k\right)  nx}{\left(  n+\beta+1\right)  ^{2}(n+k)}\nonumber \\
&  +\frac{3\alpha^{2}+3\alpha+1}{3\left(  n+\beta+1\right)  ^{2}%
},\label{testkant3}\\
K_{n}^{\left(  \alpha,\beta,k\right)  }\left(  e_{3};x\right)   &
=\frac{(n-1)(n-2)n^{3}x^{3}}{\left(  n+\beta+1\right)  ^{3}(n+k)\left(
n+2k\right)  }+\frac{3n^{2}\left(  n-1\right)  \left(  \left(  3+2\alpha
+2k\right)  n+2\left(  1+2\alpha \right)  k\right)  x^{2}}{2\left(
n+\beta+1\right)  ^{3}(n+k)\left(  n+2k\right)  }\nonumber \\
&  +\left(  \frac{4\left(  1+3\alpha(1+\alpha)\right)  k^{2}n+6k\left(
2+k+\alpha(5+3\alpha+2k\right)  )n^{2}}{2\left(  n+\beta+1\right)
^{3}(n+k)(n+2k)}\right. \nonumber \\
&  \left.  \frac{+\left(  7+6\alpha^{2}+6\alpha(2+k)+k(9+4k)\right)  n^{3}%
}{2\left(  n+\beta+1\right)  ^{3}(n+k)(n+2k)}\right)  x+\frac{4\alpha
^{3}+6\alpha^{2}+4\alpha+1}{4\left(  n+\beta+1\right)  ^{3}} \label{testkant4}%
\end{align}
\begin{align}
K_{n}^{\left(  \alpha,\beta,k\right)  }\left(  e_{4};x\right)   &
=\frac{(n-1)(n-2)\not (  n-3)n^{4}x^{4}}{\left(  n+\beta+1\right)
^{4}(n+k)(n+2k)(n+3k)}\nonumber \\
&  +\frac{2n^{3}\left(  n-1\right)  \left(  n-2\right)  \left(  \left(
4+2\alpha+3k\right)  n+3\left(  1+2a\right)  k\right)  x^{3}}{\left(
n+\beta+1\right)  ^{4}\not (  n+k)(n+2k)(n+3k)}\nonumber \\
&  +\left(  \frac{-12\left(  1+3\alpha(1+\alpha)\right)  k^{2}n^{2}%
+k(-27-5k+6\alpha(-11+\alpha(-5+6k)))n^{3}}{\left(  n+\beta+1\right)
^{4}(n+k)(n+2k)(n+3k)}\right. \nonumber \\
\end{align}
\begin{align}
&  \left.  \frac{3\left(  -5-2\alpha(3+\alpha)+k+2\alpha(9+5\alpha
)k+2(1+6\alpha)k^{2}\right)  n^{4}}{\left(  n+\beta+1\right)  ^{4}%
(n+k)(n+2k)(n+3k)}\right. \nonumber \\
&  \left.  \frac{\left(  15+6\alpha^{2}+6\alpha(3+2k)+k(24+11k)\right)  n^{5}%
}{\left(  n+\beta+1\right)  ^{4}(n+k)(n+2k)(n+3k)}\right)  x^{2}\nonumber \\
&  \left(  +\frac{6(1+2\alpha(2+\alpha(3+2\alpha)))k^{3}n+k^{2}(23+12k+2\alpha
(40+51\alpha+22\alpha^{2}+18(1+\alpha)k))n^{2}}{\left(  n+\beta+1\right)
^{4}(n+k)(n+2k)(n+3k)}\right. \nonumber \\
&  +\frac{3k(7+22\alpha+22\alpha^{2}+8\alpha^{3}+9k+22\alpha k+10\alpha
^{2}k+4k^{2}+8\alpha k^{2})n^{3}}{\left(  n+\beta+1\right)  ^{4}%
(n+k)(n+2k)(n+3k)}\nonumber \\
&  \left.  +\frac{\left(  6+2\alpha(7+2\alpha(3+\alpha))+15k+6\alpha
(3+\alpha)k+8(2+\alpha)k^{2}+6k^{3}\right)  n^{4}}{\left(  n+\beta+1\right)
^{4}(n+k)(n+2k)(n+3k)}\right)  x\nonumber \\
&  +\frac{5\alpha^{4}+10\alpha^{3}+10\alpha^{2}+5\alpha+1}{5\left(
n+\beta+1\right)  ^{4}} \label{testkant5}%
\end{align}

\begin{proof}
Following the same procedure as in Lemma \ref{test} and by direct calculations, we can complete proof of the theorem easily. Now,
we recall the moments of $P_{n,k}^{\left \langle \frac{k}{n}\right \rangle }$which are
indicated in Lemma \ref{test}. For $f\left(
t\right)  =e_{0}\left(  t\right)  ,$ we have%
\begin{align*}
K_{n}^{\left(  \alpha,\beta,k\right)  }\left(  e_{0};x\right)   &
=\frac{n+\beta+1}{\left(  n\right)  _{n,k}}\sum_{m=0}^{n}\left(
\begin{array}
[c]{c}%
n\\
m
\end{array}
\right)  \left(  nx\right)  _{m,k}\left(  n-nx\right)  _{n-m,k}\int
\limits_{\frac{m+\alpha}{n+\beta+1}}^{\frac{m+\alpha+1}{n+\beta+1}}dt\\
&  =\frac{n+\beta+1}{\left(  n\right)  _{n,k}}\sum_{m=0}^{n}\left(
\begin{array}
[c]{c}%
n\\
m
\end{array}
\right)  \left(  nx\right)  _{m,k}\left(  n-nx\right)  _{n-m,k}\left(
\frac{1}{n+\beta+1}\right) \\
&  =\frac{1}{\left(  n\right)  _{n,k}}\sum_{m=0}^{n}\left(
\begin{array}
[c]{c}%
n\\
m
\end{array}
\right)  \left(  nx\right)  _{m,k}\left(  n-nx\right)  _{n-m,k}\\
&  =P_{n,k}^{\left \langle \frac{k}{n}\right \rangle }\left(  e_{0};x\right) \\
&  =1
\end{align*}
For the case $f\left(  t\right)  =e_{1}\left(  t\right)  ,\ $it follows%
\begin{align*}
K_{n}^{\left(  \alpha,\beta,k\right)  }\left(  e_{1};x\right)   &
=\frac{n+\beta+1}{\left(  n\right)  _{n,k}}\sum_{m=0}^{n}\left(
\begin{array}
[c]{c}%
n\\
m
\end{array}
\right)  \left(  nx\right)  _{m,k}\left(  n-nx\right)  _{n-m,k}\int
\limits_{\frac{m+\alpha}{n+\beta+1}}^{\frac{m+\alpha+1}{n+\beta+1}}tdt\\
&  =\frac{n+\beta+1}{2\left(  n\right)  _{n,k}}\sum_{m=0}^{n}\left(
\begin{array}
[c]{c}%
n\\
m
\end{array}
\right)  \left(  nx\right)  _{m,k}\left(  n-nx\right)  _{n-m,k}\left(
\frac{\left(  m+\alpha+1\right)  ^{2}-\left(  m+\alpha \right)  ^{2}}{\left(
n+\beta+1\right)  ^{2}}\right) \\
&  =\frac{2\alpha+1}{2\left(  n+\beta+1\right)  }\frac{1}{\left(  n\right)
_{n,k}}\sum_{m=0}^{n}\left(
\begin{array}
[c]{c}%
n\\
m
\end{array}
\right)  \left(  nx\right)  _{m,k}\left(  n-nx\right)  _{n-m,k}\\
&  +\frac{n}{n+\beta+1}\frac{1}{\left(  n\right)  _{n,k}}\sum_{m=0}^{n}\left(
\begin{array}
[c]{c}%
n\\
m
\end{array}
\right)  \left(  nx\right)  _{m,k}\left(  n-nx\right)  _{n-m,k}\left(
\frac{m}{n}\right) 
\end{align*}
\begin{align*}
&  =\frac{n}{n+\beta+1}P_{n,k}^{\left \langle \frac{k}{n}\right \rangle }\left(
e_{1};x\right)  +\frac{2\alpha+1}{2\left(  n+\beta+1\right)  }P_{n,k}%
^{\left \langle \frac{k}{n}\right \rangle }\left(  e_{0};x\right) \\
&  =\frac{2\alpha+1}{2\left(  n+\beta+1\right)  }+\frac{n}{n+\beta+1}x
\end{align*}
Taking into account $P_{n,k}^{\left \langle \frac{k}{n}\right \rangle }\left(
e_{i};.\right)  \ i=0,1,2$, we can write
\begin{align*}
K_{n}^{\left(  \alpha,\beta,k\right)  }\left(  e_{2};x\right)   &
=\frac{n+\beta+1}{\left(  n\right)  _{n,k}}\sum_{m=0}^{n}\left(
\begin{array}
[c]{c}%
n\\
m
\end{array}
\right)  \left(  nx\right)  _{m,k}\left(  n-nx\right)  _{n-m,k}\int
\limits_{\frac{m+\alpha}{n+\beta+1}}^{\frac{m+\alpha+1}{n+\beta+1}}t^{2}dt\\
&  =\frac{1}{3\left(  n+\beta+1\right)  ^{2}\left(  n\right)  _{n,k}}%
\sum_{m=0}^{n}\left(
\begin{array}
[c]{c}%
n\\
m
\end{array}
\right)  \left(  nx\right)  _{m,k}\left(  n-nx\right)  _{n-m,k}\left(  \left(
m+\alpha+1\right)  ^{3}-\left(  m+\alpha \right)  ^{3}\right)\\
&  =\frac{n^{2}}{\left(  n+\beta+1\right)  ^{2}}\frac{1}{\left(  n\right)
_{n,k}}\sum_{m=0}^{n}\left(
\begin{array}
[c]{c}%
n\\
m
\end{array}
\right)  \left(  nx\right)  _{m,k}\left(  n-nx\right)  _{n-m,k}\left(
\frac{m^{2}}{n^{2}}\right) \\
&  +\frac{\left(  2\alpha+1\right)  n}{\left(  n+\beta+1\right)  ^{2}}\frac
{1}{\left(  n\right)  _{n,k}}\sum_{m=0}^{n}\left(
\begin{array}
[c]{c}%
n\\
m
\end{array}
\right)  \left(  nx\right)  _{m,k}\left(  n-nx\right)  _{n-m,k}\left(
\frac{m}{n}\right) \\
&  +\frac{\left(  \alpha+1\right)  ^{3}-\alpha^{3}}{3\left(  n+\beta+1\right)
^{2}}\frac{1}{\left(  n\right)  _{n,k}}\sum_{m=0}^{n}\left(
\begin{array}
[c]{c}%
n\\
m
\end{array}
\right)  \left(  nx\right)  _{m,k}\left(  n-nx\right)  _{n-m,k}\\
&  =\frac{n^{2}}{\left(  n+\beta+1\right)  ^{2}}P_{n,k}^{\left \langle \frac
{k}{n}\right \rangle }\left(  e_{2};x\right)  +\frac{\left(  2\alpha+1\right)
n}{\left(  n+\beta+1\right)  ^{2}}P_{n,k}^{\left \langle \frac{k}%
{n}\right \rangle }\left(  e_{1};x\right) \\
&  +\frac{\left(  \alpha+1\right)  ^{3}-\alpha^{3}}{3\left(  n+\beta+1\right)
^{2}}P_{n,k}^{\left \langle \frac{k}{n}\right \rangle }\left(  e_{0};x\right) \\
&  =\frac{n^{2}}{\left(  n+\beta+1\right)  ^{2}}\left \{  x^{2}+\frac{k+1}%
{n+k}x\left(  1-x\right)  \right \}  +\frac{\left(  2\alpha+1\right)
n}{\left(  n+\beta+1\right)  ^{2}}x+\frac{\left(  \alpha+1\right)  ^{3}%
-\alpha^{3}}{3\left(  n+\beta+1\right)  ^{2}}%
\end{align*}
The proof for $f(t)=e_{i}\left(  t\right)  ,\ i=3,4$ is quite similar as
others, hence results are given as follows
\begin{align*}
K_{n}^{\left(  \alpha,\beta,k\right)  }\left(  e_{3};x\right)   &
=\frac{n+\beta+1}{\left(  n\right)  _{n,k}}\sum_{m=0}^{n}\left(
\begin{array}
[c]{c}%
n\\
m
\end{array}
\right)  \left(  nx\right)  _{m,k}\left(  n-nx\right)  _{n-m,k}\int
\limits_{\frac{m+\alpha}{n+\beta+1}}^{\frac{m+\alpha+1}{n+\beta+1}}t^{3}dt\\
&  =\frac{1}{4\left(  n+\beta+1\right)  ^{3}}\left \{  4n^{3}P_{n,k}%
^{\left \langle \frac{k}{n}\right \rangle }\left(  e_{3};x\right)  +6\left(
2\alpha+1\right)  n^{2}P_{n,k}^{\left \langle \frac{k}{n}\right \rangle }\left(
e_{2};x\right)  \right. \\
&  \left.  +4\left(  3\alpha^{2}+3\alpha+1\right)  nP_{n,k}^{\left \langle
\frac{k}{n}\right \rangle }\left(  e_{1};x\right)  +\left(  \left(
\alpha+1\right)  ^{4}-\alpha^{4}\right)  P_{n,k}^{\left \langle \frac{k}%
{n}\right \rangle }\left(  e_{0};x\right)  \right \}
\end{align*}
and
\begin{align*}
K_{n}^{\left(  \alpha,\beta,k\right)  }\left(  e_{4};x\right)   &
=\frac{n+\beta+1}{\left(  n\right)  _{n,k}}\sum_{m=0}^{n}\left(
\begin{array}
[c]{c}%
n\\
m
\end{array}
\right)  \left(  nx\right)  _{m,k}\left(  n-nx\right)  _{n-m,k}\int
\limits_{\frac{m+\alpha}{n+\beta+1}}^{\frac{m+\alpha+1}{n+\beta+1}}t^{4}dt\\
&  =\frac{1}{5\left(  n+\beta+1\right)  ^{4}}\left \{  5n^{4}P_{n,k}%
^{\left \langle \frac{k}{n}\right \rangle }\left(  e_{4};x\right)  +10\left(
2\alpha+1\right)  n^{3}P_{n}^{\left \langle \frac{k}{n}\right \rangle }\left(
e_{3};x\right)  \right. \\
&  \left.  +10\left(  3\alpha^{2}+3\alpha+1\right)  n^{2}P_{n,k}^{\left \langle
\frac{k}{n}\right \rangle }\left(  e_{2};x\right)  \right. \\
&  \left.  +5\left(  4\alpha^{3}+6\alpha^{2}+4\alpha+1\right)  nP_{n,k}%
^{\left \langle \frac{k}{n}\right \rangle }\left(  e_{1};x\right)  +\left(
\left(  \alpha+1\right)  ^{5}-\alpha^{5}\right)  P_{n,k}^{\left \langle
\frac{k}{n}\right \rangle }\left(  e_{0};x\right)  \right \}  ,
\end{align*}
by taking into account $P_{n,k}^{\left \langle \frac{k}{n}\right \rangle
}\left(  e_{i};.\right)  \ i=0,1,2,3,4$, we obtain the desired results.
\end{proof}

\begin{corollary}
\label{centkant} Let $n\in%
\mathbb{N}
$,$\ k\geq0.$ Then the central moments of the operators $K_{n}^{\left(
\alpha,\beta,k\right)  }$ are given by%
\begin{align}
K_{n}^{\left(  \alpha,\beta,k\right)  }\left(  e_{1}-x;x\right)   &
=\frac{2\alpha+1}{2\left(  n+\beta+1\right)  }-\frac{\beta+1}{n+\beta
+1}x,\label{centkant1}\\
K_{n}^{\left(  \alpha,\beta,k\right)  }\left(  \left(  e_{1}-x\right)
^{2};x\right)   &  =\frac{\left(  \left(  \beta+1\right)  ^{2}\left(
n+k\right)  -\left(  k+1\right)  n^{2}\right)  x^{2}}{\left(  n+\beta
+1\right)  ^{2}\left(  n+k\right)  }\nonumber \\
&  +\frac{\left(  \left(  k+1\right)  n^{2}-\left(  1+2\alpha \right)  \left(
\beta+1\right)  \left(  n+k\right)  \right)  x}{\left(  n+\beta+1\right)
^{2}\left(  n+k\right)  }\nonumber \\
&  +\frac{3\alpha^{2}+3\alpha+1}{3\left(  n+\beta+1\right)  ^{2}}.
\label{centkant2}%
\end{align}
Moreover $K_{n}^{\left(  \alpha,\beta,k\right)  }\left(  \left(
e_{1}-x\right)  ^{4};x\right)  =O\left(  \frac{1}{n^{2}}\right)  $ as
$n\rightarrow0.$
\end{corollary}

\begin{proof}
Exploiting the previous results and doing some simple computations allow the
proof of the corollary.
\end{proof}

\begin{lemma}
\label{m} For $n\in%
\mathbb{N}
$ and $k\geq0,$ we have%
\[
K_{n}^{\left(  \alpha,\beta,k\right)  }\left(  \left(  e_{1}-x\right)
^{2};x\right)  \leq \xi_{n,k}^{\alpha,\beta}%
\]
where $\xi_{n,k}^{\alpha,\beta}=\frac{1}{n+\beta+1}\left \{  \frac{k+1}%
{4}+\beta+2\alpha+\frac{\left(  \alpha+1\right)  ^{3}-\alpha^{3}}{3}\right \}
.$
\end{lemma}

\begin{proof}
From Corollary \ref{centkant}, it follows%
\begin{align*}
K_{n}^{\left(  \alpha,\beta,k\right)  }\left(  \left(  e_{1}-x\right)
^{2};x\right)   &  =\frac{1}{\left(  n+\beta+1\right)  ^{2}}\left \{  x\left(
1-x\right)  \left(  \frac{n^{2}\left(  k+1\right)  }{n+k}-\left(
\beta+1\right)  ^{2}\right)  \right. \\
&  +\left.  \left(  \beta+1\right)  \left(  \beta-2\alpha \right)
x+\frac{\left(  \alpha+1\right)  ^{3}-\alpha^{3}}{3}\right \} \\
&  \leq \frac{1}{n+\beta+1}\left \{  \left(  k+1\right)  x\left(  1-x\right)
+\frac{\left(  \beta+1\right)  }{n+\beta+1}\left(  \beta+2\alpha \right)
x+\frac{\left(  \alpha+1\right)  ^{3}-\alpha^{3}}{3\left(  n+\beta+1\right)
}\right \} \\
&  \leq \frac{1}{n+\beta+1}\left \{  \left(  k+1\right)  x\left(  1-x\right)
+\left(  \beta+2\alpha \right)  +\frac{\left(  \alpha+1\right)  ^{3}-\alpha
^{3}}{3}\right \} \\
&  \leq \frac{1}{n+\beta+1}\left \{  \frac{k+1}{4}+\beta+2\alpha+\frac{\left(
\alpha+1\right)  ^{3}-\alpha^{3}}{3}\right \}  .
\end{align*}

\end{proof}

\begin{corollary}
\label{t} Taking into account Corollary \ref{centkant}, we get
the following limits as follows,
\begin{align}
\lim_{n\rightarrow \infty}nK_{n}^{\left(  \alpha,\beta,k\right)  }\left(
e_{1}-x;x\right)   &  =\alpha+\frac{1}{2}-\left(  \beta+1\right)
x,\label{limkant1}\\
\lim_{n\rightarrow \infty}nK_{n}^{\left(  \alpha,\beta,k\right)  }\left(
\left(  e_{1}-x\right)  ^{2};x\right)   &  =\left(  k+1\right)  x\left(
1-x\right)  ,\label{limkant2}\\
\lim_{n\rightarrow \infty}n^{2}K_{n}^{\left(  \alpha,\beta,k\right)  }\left(
\left(  e_{1}-x\right)  ^{4};x\right)   &  =3\left(  k+1\right)  ^{2}%
x^{2}\left(  1-x\right)  ^{2}. \label{limkant3}%
\end{align}

\end{corollary}

\subsection{Convergence Properties of $K_{n}^{\left(  \alpha,\beta,k\right)
}$}

\begin{theorem}
Let $n\in%
\mathbb{N}
$,$\ k\geq0.\ $Then for every $f\in C\left[  0,1\right]  ,$ we have%
\begin{equation}
\lim_{n\rightarrow \infty}K_{n}^{\left(  \alpha,\beta,k\right)  }\left(
f;x\right)  =f\left(  x\right)  \label{korkant}%
\end{equation}
uniformly in $\left[  0,1\right]  .$
\end{theorem}

\begin{proof}
Making use of the results in Lemma \ref{testkant}, we deduce that
\[
\lim_{n\rightarrow \infty}K_{n}^{\left(  \alpha,\beta,k\right)  }\left(
e_{i}(t);x\right)  =x^{i}\  \ ,\  \ i=0,1,2
\]
uniformly in $\left[  0,1\right]  .$ According to Korovkin's theorem, one can
easily get the desired result.
\end{proof}

Now, we are in a position to give the theorems of the rate of convergence of
the operators by virtue of classical modulus of continuity.

For $\delta>0,$ the modulus of continuity of $f$ denoted by $w\left(
f;\delta \right)  $ is defined to be
\begin{equation}
\omega \left(  f;\delta \right)  =\sup_{\substack{x,t\in \left[  a,b\right]
\\ \left \vert t-x\right \vert \leq \delta}}\left \vert f\left(  t\right)
-f\left(  x\right)  \right \vert . \label{mod}%
\end{equation}

Then$,$ for any $\delta>0$ and each $x\in \left[  a,b\right]  $, it is well
known that%
\begin{equation}
\left \vert f\left(  t\right)  -f\left(  x\right)  \right \vert \leq \left(
\frac{\left \vert t-x\right \vert }{\delta}+1\right)  \omega \left(
f;\delta \right)  . \label{mod1}%
\end{equation}

\begin{theorem}
Let $n\in%
\mathbb{N}
$,$\ k\geq0.\ $Then for every $f\in C\left[  0,1\right]  ,$ we have the
following result%
\begin{equation}
\left \vert K_{n}^{\left(  \alpha,\beta,k\right)  }\left(  f;x\right)
-f\left(  x\right)  \right \vert \leq2\omega \left(  f;\sqrt{K_{n}^{\left(
\alpha,\beta,k\right)  }\left(  \left(  e_{1}-x\right)  ^{2};x\right)
}\right)  , \label{modkant}%
\end{equation}
where $\omega \left(  f;.\right)  $ is modulus of continuity defined by
$\left(  \text{\ref{mod}}\right)  $.
\end{theorem}

\begin{proof}
In virtue of the definition of the operators $K_{n}^{\left(  \alpha
,\beta,k\right)  }$given by $\left(  \text{\ref{PolKant}}\right)  $ and by the
property of modulus of continuity $\left(  \text{\ref{mod1}}\right)  $, we
obtain
\begin{align}
&  \left \vert K_{n}^{\left(  \alpha,\beta,k\right)  }\left(  f;x\right)
-f\left(  x\right)  \right \vert \nonumber \\
&  \leq \frac{n+\beta+1}{\left(  n\right)  _{n,k}}\sum_{m=0}^{n}\left(
\begin{array}
[c]{c}%
n\\
m
\end{array}
\right)  \left(  nx\right)  _{m,k}\left(  n-nx\right)  _{n-m,k}\int
\limits_{\frac{m+\alpha}{n+\beta+1}}^{\frac{m+\alpha+1}{n+\beta+1}}\left \vert
f\left(  t\right)  -f\left(  x\right)  \right \vert dt\nonumber \\
&  \leq \frac{n+\beta+1}{\left(  n\right)  _{n,k}}\sum_{m=0}^{n}\left(
\begin{array}
[c]{c}%
n\\
m
\end{array}
\right)  \left(  nx\right)  _{m,k}\left(  n-nx\right)  _{n-m,k}\int
\limits_{\frac{m+\alpha}{n+\beta+1}}^{\frac{m+\alpha+1}{n+\beta+1}}\left(
\frac{\left \vert t-x\right \vert }{\delta}+1\right)  \omega \left(
f;\delta \right)  dt\nonumber \\
&  =\left[  \frac{1}{\delta}\left(  \frac{n+\beta+1}{\left(  n\right)  _{n,k}%
}\sum_{m=0}^{n}\left(
\begin{array}
[c]{c}%
n\\
m
\end{array}
\right)  \left(  nx\right)  _{m,k}\left(  n-nx\right)  _{n-m,k}\int
\limits_{\frac{m+\alpha}{n+\beta+1}}^{\frac{m+\alpha+1}{n+\beta+1}}\left \vert
t-x\right \vert dt\right)  +1\right]  \omega \left(  f;\delta \right)  .
\label{modkant1}%
\end{align}
Using Cauchy Schwarz inequality, the integral can be written as follows%
\begin{equation}
\int \limits_{\frac{m+\alpha}{n+\beta+1}}^{\frac{m+\alpha+1}{n+\beta+1}%
}\left \vert t-x\right \vert dt\leq \frac{1}{\sqrt{n+\beta+1}}\left(
\int \limits_{\frac{m+\alpha}{n+\beta+1}}^{\frac{m+\alpha+1}{n+\beta+1}}\left(
t-x\right)  ^{2}dt\right)  ^{\frac{1}{2}}. \label{modkant2}%
\end{equation}
Taking into consideration again Cauchy Schwarz inequality for summation and
combining it with $\left(  \text{\ref{modkant2}}\right)  $, we reach
\begin{align}
&  \sum_{m=0}^{n}\left(
\begin{array}
[c]{c}%
n\\
m
\end{array}
\right)  \left(  nx\right)  _{m,k}\left(  n-nx\right)  _{n-m,k}\int
\limits_{\frac{m+\alpha}{n+\beta+1}}^{\frac{m+\alpha+1}{n+\beta+1}}\left \vert
t-x\right \vert dt\nonumber \\
&  \leq \sum_{m=0}^{n}\left(
\begin{array}
[c]{c}%
n\\
m
\end{array}
\right)  \left(  nx\right)  _{m,k}\left(  n-nx\right)  _{n-m,k}\left(
\frac{1}{n+\beta+1}\int \limits_{\frac{m+\alpha}{n+\beta+1}}^{\frac{m+\alpha
+1}{n+\beta+1}}\left(  t-x\right)  ^{2}dt\right)  ^{1/2}\nonumber \\
&  \leq \frac{\left(  n\right)  _{n,k}}{n+\beta+1}\left(  K_{n}^{\left(
\alpha,\beta,k\right)  }\left(  \left(  e_{1}-x\right)  ^{2};x\right)
\right)  ^{1/2}. \label{modkant3}%
\end{align}
Finally, in view of the $\left(  \text{\ref{modkant3}}\right)  ,\  \left(
\text{\ref{modkant1}}\right)  $ can be expressed as
\[
\left \vert K_{n}^{\left(  \alpha,\beta,k\right)  }\left(  f;x\right)
-f\left(  x\right)  \right \vert \leq \left[  \frac{1}{\delta}\left(
K_{n}^{\left(  \alpha,\beta,k\right)  }\left(  \left(  e_{1}-x\right)
^{2};x\right)  \right)  ^{1/2}+1\right]  \omega \left(  f;\delta \right)  .
\]
Choosing $\delta=\left(  K_{n}^{\left(  \alpha,\beta,k\right)  }\left(
\left(  e_{1}-x\right)  ^{2};x\right)  \right)  ^{1/2},\ $we get desired
result as%
\[
\left \vert K_{n}^{\left(  \alpha,\beta,k\right)  }\left(  f;x\right)
-f\left(  x\right)  \right \vert \leq2\omega \left(  f;\sqrt{K_{n}^{\left(
\alpha,\beta,k\right)  }\left(  \left(  e_{1}-x\right)  ^{2};x\right)
}\right)  .
\]

\end{proof}

\begin{theorem}
If $f\in C^{1}\left[  0,1\right]  ,$ then
\[
\left \vert K_{n}^{\left(  \alpha,\beta,k\right)  }\left(  f;x\right)
-f\left(  x\right)  \right \vert \leq \nu_{1}\left \vert f^{\prime}\left(
x\right)  \right \vert +2\sqrt{\nu_{2}}~\omega \left(  f^{\prime};\sqrt{\nu_{2}%
}\right)
\]
where $\nu_{1}=K_{n}^{\left(  \alpha,\beta,k\right)  }\left(  e_{1}%
-x;x\right)  \ $and $\nu_{2}=K_{n}^{\left(  \alpha,\beta,k\right)  }\left(
\left(  e_{1}-x\right)  ^{2};x\right)  .$
\end{theorem}

\begin{proof}
Let $f\in C^{1}\left[  0,1\right]  .$ For any $x,t\in \left[  0,1\right]  $%
\[
f\left(  t\right)  -f\left(  x\right)  =f^{\prime}\left(  x\right)  \left(
t-x\right)  +%
{\displaystyle \int \limits_{x}^{t}}
\left(  f^{\prime}\left(  y\right)  -f^{\prime}\left(  x\right)  \right)  dy.
\]
Applying $K_{n}^{\left(  \alpha,\beta,k\right)  }$ on the both side%
\[
K_{n}^{\left(  \alpha,\beta,k\right)  }\left(  f\left(  t\right)  -f\left(
x\right)  ;x\right)  =f^{\prime}\left(  x\right)  K_{n}^{\left(  \alpha
,\beta,k\right)  }\left(  e_{1}-x;x\right)  +K_{n}^{\left(  \alpha
,\beta,k\right)  }\left(  \left(
{\displaystyle \int \limits_{x}^{t}}
\left(  f^{\prime}\left(  y\right)  -f^{\prime}\left(  x\right)  \right)
dy\right)  ;x\right)  .
\]
By the property of modulus continuity (\ref{mod1}), we get%
\[
\left \vert
{\displaystyle \int \limits_{x}^{t}}
\left \vert f^{\prime}\left(  y\right)  -f^{\prime}\left(  x\right)
\right \vert dy\right \vert \leq \omega \left(  f^{\prime};\delta \right)  \left(
\frac{\left(  t-x\right)  ^{2}}{\delta}+\left \vert t-x\right \vert \right)  ,
\]
from which we conclude that%
\begin{align*}
\left \vert K_{n}^{\left(  \alpha,\beta,k\right)  }\left(  f;x\right)
-f\left(  x\right)  \right \vert  &  \leq \left \vert f^{\prime}\left(  x\right)
\right \vert \left \vert K_{n}^{\left(  \alpha,\beta,k\right)  }\left(
e_{1}-x;x\right)  \right \vert \\
&  +\left[  \frac{1}{\delta}K_{n}^{\left(  \alpha,\beta,k\right)  }\left(
\left(  e_{1}-x\right)  ^{2};x\right)  +K_{n}^{\left(  \alpha,\beta,k\right)
}\left(  \left \vert e_{1}-x\right \vert ;x\right)  \right]  \omega \left(
f^{\prime};\delta \right)  .
\end{align*}
By the Cauchy Schwarz inequality, we can write
\begin{align*}
&  \left \vert K_{n}^{\left(  \alpha,\beta,k\right)  }\left(  f;x\right)
-f\left(  x\right)  \right \vert \\
&  \leq \left \vert f^{\prime}\left(  x\right)  \right \vert \left \vert
K_{n}^{\left(  \alpha,\beta,k\right)  }\left(  e_{1}-x;x\right)  \right \vert
\\
&  +\left[  \frac{1}{\delta}\left(  K_{n}^{\left(  \alpha,\beta,k\right)
}\left(  \left(  e_{1}-x\right)  ^{2};x\right)  \right)  ^{1/2}+1\right]
\left(  K_{n}^{\left(  \alpha,\beta,k\right)  }\left(  \left(  e_{1}-x\right)
^{2};x\right)  \right)  ^{1/2}\omega \left(  f^{\prime};\delta \right)  .
\end{align*}
By taking $K_{n}^{\left(  \alpha,\beta,k\right)  }\left(  e_{1}-x;x\right)
=\nu_{1}\ $and $K_{n}^{\left(  \alpha,\beta,k\right)  }\left(  \left(
e_{1}-x\right)  ^{2};x\right)  =\nu_{2}$, and by choosing $\delta=\sqrt
{\nu_{2}}$, we arrive at the desired result.
\end{proof}

\begin{theorem}
\label{vo} Let $f$ be bounded on $\left[  0,1\right]  .$ Then for any $x\in \left(
0,1\right)  $ at which $f^{\prime},f^{\prime \prime}$ exist,\ we have%
\begin{equation}
\lim_{n\rightarrow \infty}n\left[  K_{n}^{\left(  \alpha,\beta,k\right)
}\left(  f;x\right)  -f\left(  x\right)  \right]  =\frac{1}{2}\left[  \left(
2\alpha+1-2\left(  \beta+1\right)  x\right)  f^{\prime}(x)+\left(  k+1\right)
x\left(  1-x\right)  f^{\prime \prime}(x)\right]  . \label{Vorokant}%
\end{equation}

\end{theorem}

\begin{proof}
Let $x\in \left[  0,1\right]  .$ According to Taylor's series expansion of the
function $f$ at the point $x,$we can write%
\begin{equation}
f(s)=f(x)+f^{\prime}(x)(s-x)+\frac{1}{2}f^{\prime \prime}(x)(s-x)^{2}%
+\eta(s,x)(s-x)^{2} \label{Taylorkant}%
\end{equation}
where $\eta(s,x)\in C[0,1]\ $which satisfies%
\begin{equation}
\lim_{s\rightarrow x}\eta(s,x)=0. \label{remain}%
\end{equation}
Employing the operators $K_{n}^{\left(  \alpha,\beta,k\right)  }$ to the both
sides of $\left(  \text{\ref{Taylorkant}}\right)  $ and taking limit for
$n\rightarrow \infty,$ it follows that%
\begin{align*}
\lim_{n\rightarrow \infty}n\left(  K_{n}^{\left(  \alpha,\beta,k\right)
}\left(  f;x\right)  -f\left(  x\right)  \right)   &  =f^{\prime}%
(x)\lim_{n\rightarrow \infty}nK_{n}^{\left(  \alpha,\beta,k\right)  }\left(
s-x;x\right) \\
&  +\frac{1}{2}f^{\prime \prime}\left(  x\right)  \lim_{n\rightarrow \infty
}nK_{n}^{\left(  \alpha,\beta,k\right)  }\left(  \left(  s-x\right)
^{2};x\right) \\
&  +\lim_{n\rightarrow \infty}nK_{n}^{\left(  \alpha,\beta,k\right)  }\left(
\eta(s,x)(s-x)^{2};x\right)  .
\end{align*}
Applying the Cauchy-Schwarz inequality for the the last term$\ K_{n}^{\left(
\alpha,\beta,k\right)  }\left(  \eta(s,x)(s-x)^{2};x\right)  $, we have%
\[
K_{n}^{\left(  \alpha,\beta,k\right)  }\left(  \eta(s,x)(s-x)^{2};x\right)
\leq \sqrt{K_{n}^{\left(  \alpha,\beta,k\right)  }\left(  \eta^{2}%
(s,x);x\right)  }\sqrt{K_{n}^{\left(  \alpha,\beta,k\right)  }\left(
(s-x)^{4};x\right)  }.
\]
From the property of $\left(  \text{\ref{remain}}\right)  $ and from the limit
$\left(  \text{\ref{limkant3}}\right)  ,\ $we obtain%
\[
\lim_{n\rightarrow \infty}nK_{n}^{\left(  \alpha,\beta,k\right)  }\left(
\eta(s,x)(s-x)^{2};x\right)  =0.
\]
Finally taking into account $\left(  \text{\ref{limkant1}}\right)  $ and
$\left(  \text{\ref{limkant2}}\right)  ,\ $we arrive at%
\[
\lim_{n\rightarrow \infty}n\left(  K_{n}^{\left(  \alpha,\beta,k\right)
}\left(  f;x\right)  -f\left(  x\right)  \right)  =\frac{1}{2}\left[  \left(
2\alpha+1-2\left(  \beta+1\right)  x\right)  f^{\prime}(x)+\left(  k+1\right)
x\left(  1-x\right)  f^{\prime \prime}(x)\right]  ,
\]
which is required result.
\end{proof}

\begin{remark}
If we get $\alpha=\beta=0,~k=1$ in Theorem \ref{vo}, we obtain the Voronovskaja type
theorem given for the operators $D_{n}^{\ast \left(  \frac{1}{n}\right)  }$ by
Agrawal et al. \cite{AIK}.
\end{remark}

Now, we continue by analizing the approximation properties of the operators
$K_{n}^{\left(  \alpha,\beta,k\right)  }\ $taking into account the graphical examples.

\begin{example}
Let $f\left(  x\right)  =x^{3}\sin \left(  4\pi x\right)  ,\ n=50,$
$\alpha=\beta=0$ and $k=0.2.\ $Behaviours of the approximation for the
modified operators $K_{n}^{\left(  \alpha,\beta,k\right)  }$, the operators
$D_{n}^{\ast \left(  \frac{1}{n}\right)  }$ which was given in \cite{AIK} and
the classical Kantorovich operators $K_{n}\ $are illustrated in Figure 4. We
choose the same function in \cite{AIK} for the better comparison and in Figure
4, one can see that for $k=0.2$, $K_{n}^{\left(  \alpha,\beta,k\right)  }$
provides a better approximation than the operators $D_{n}^{\ast \left(
\frac{1}{n}\right)  }$ to the function $f$ but the classical Kantorovich
operators $K_{n}$ have better approximation than others.
\end{example}

\begin{figure}[pth]
	\includegraphics*[width=7cm]{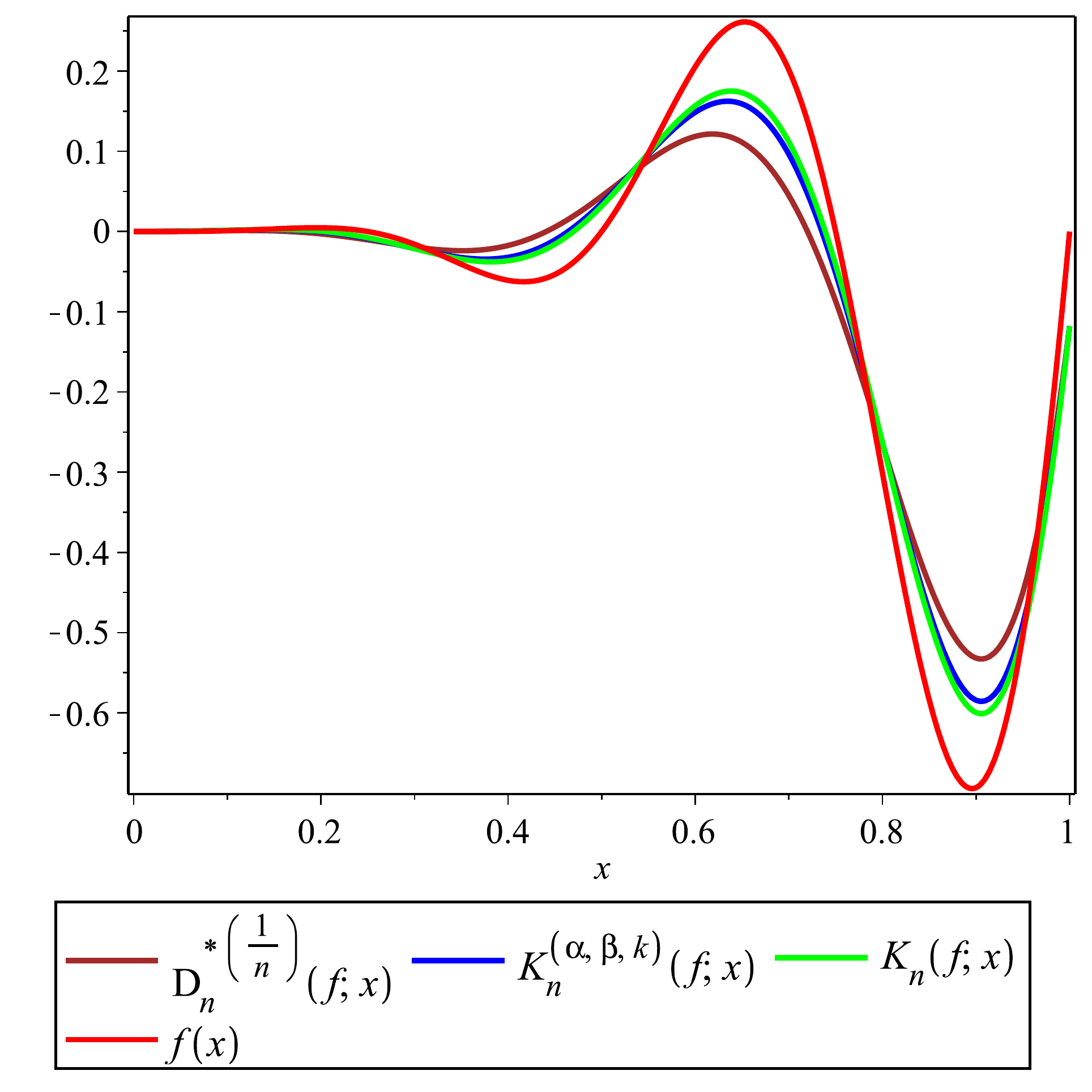}\caption{Convergence of
		$K_{n}^{\left(  \alpha,\beta,k\right)  }$, $D_{n}^{\ast \left(  \frac{1}%
			{n}\right)  }$ and $K_{n}\ $ to the function $f$ for $n=50$, $\alpha=\beta=0$
		and $k=0.2$}%
\end{figure}

\begin{example}
Consider $f\left(  x\right)  =2x^{2}\sin \left(  2\pi x\right)  $ and $k=0.3.$
Figure 5 presents the approximation process of the operators $K_{n}^{\left(
\alpha,\beta,k\right)  }$ for the special choices of $n=30,90,150$ and
$\alpha=\beta=0$.\ It is clearly seen that as the value of $n$ increases,\ the
approximation of the operators $K_{n}^{\left(  \alpha,\beta,k\right)  }\ $is
getting better.
\end{example}

\begin{figure}[pth]
\includegraphics*[width=7cm]{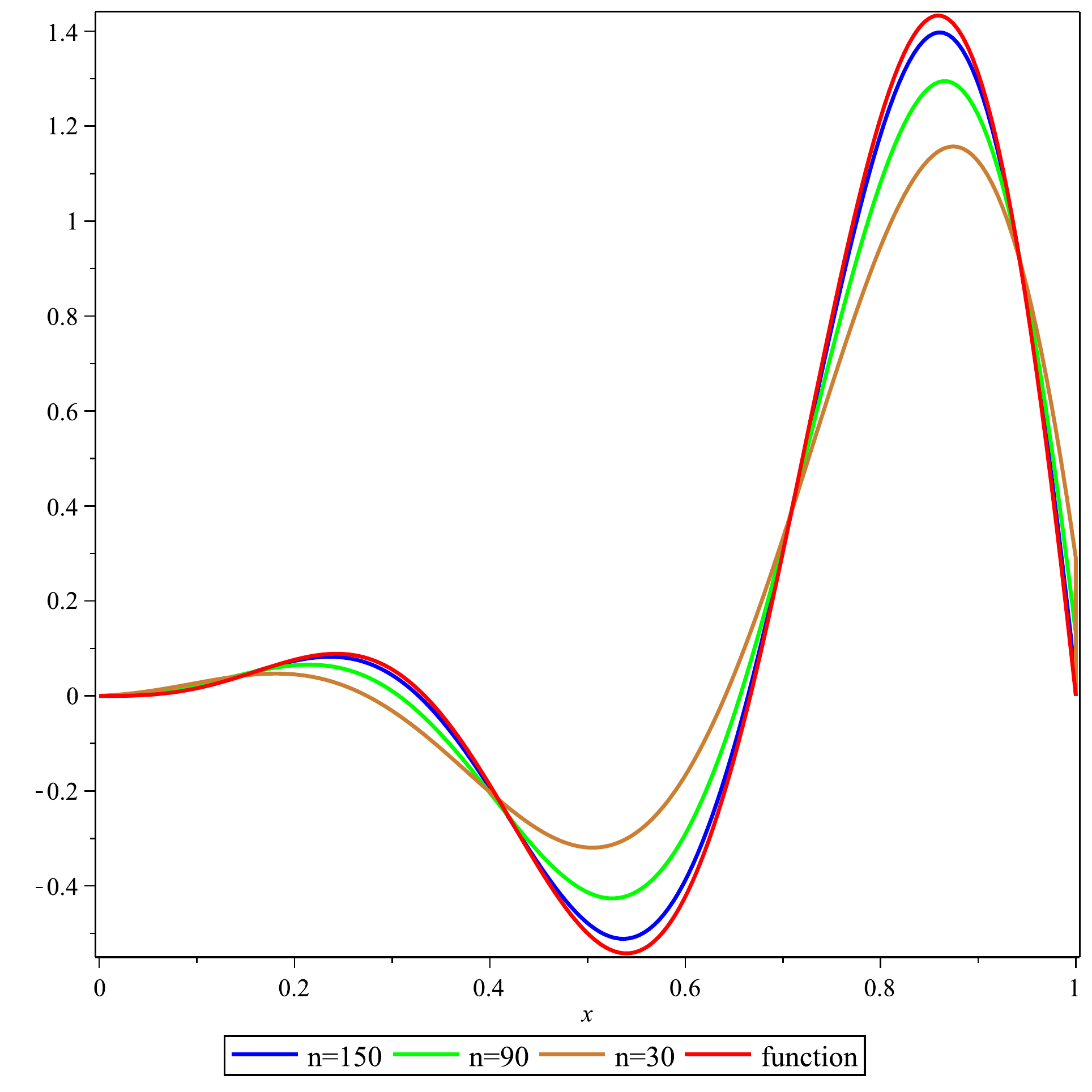}\caption{Approximation of the
operators $K_{n}^{\left(  \alpha,\beta,k\right)  }$ for $\alpha=\beta=0$ and
$k=0.3$ fixed and $n=30,90,150$}%
\end{figure}

\begin{example}
Let $f\left(  x\right)  =x^{5}\left(  x-\frac{1}{4}\right)  \sin \left(  \pi
x\right)  .$ Figure 6 shows the convergence of $K_{n}^{\left(  \alpha
,\beta,k\right)  }$ to the function $f\ $for $n=20\ $fixed and
$k=0.3,\ 0.6,\ 0.9,\ 1.2$ and $1.5$ and $\alpha=\beta=1.\ $As the value of $k$
decrease towards to zero,\ the approximation of the operators $K_{n}^{\left(
\alpha,\beta,k\right)  }\ $is getting better.
\end{example}

\begin{figure}[pth]
\includegraphics*[width=7cm]{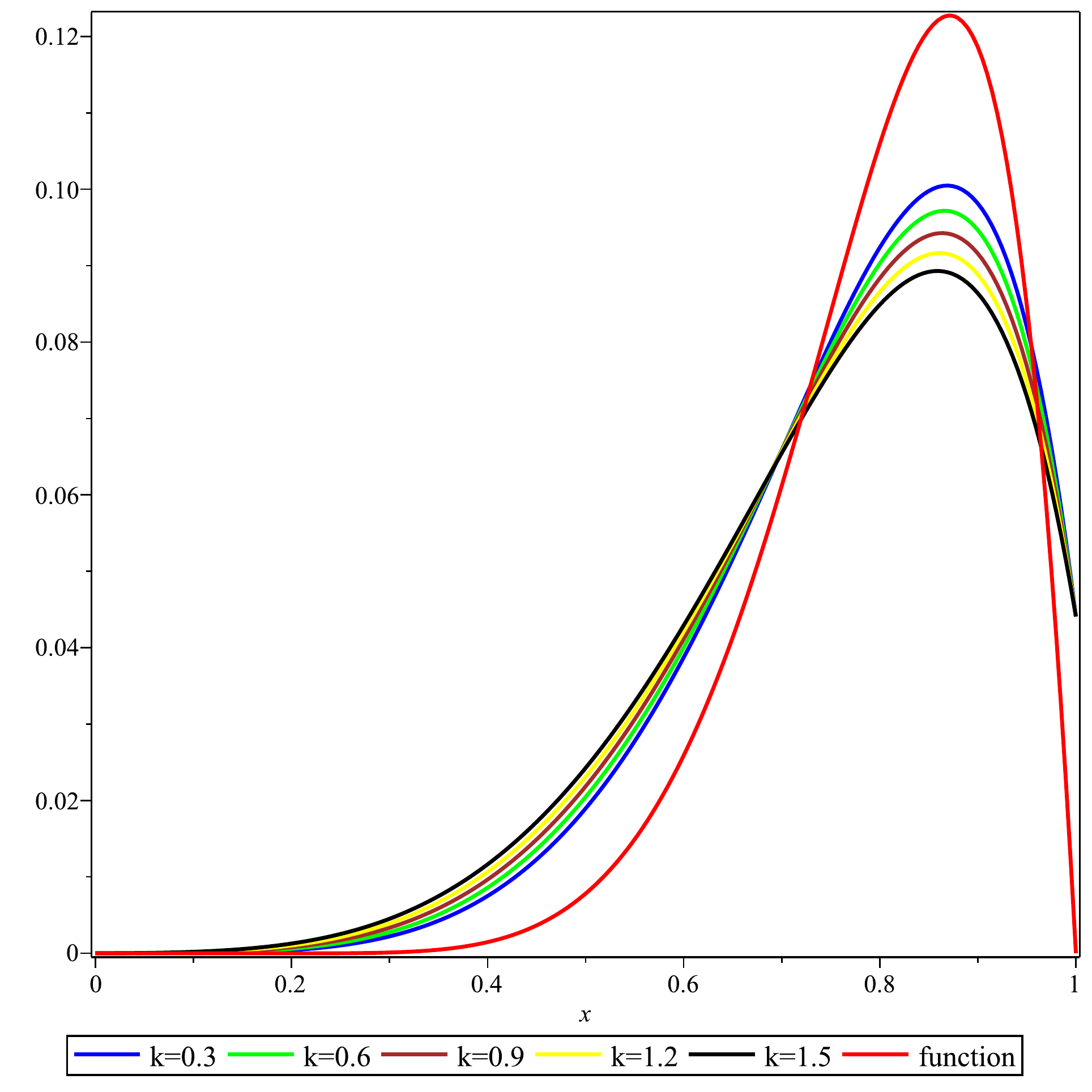}\caption{Approximation of the
operators $K_{n}^{\left(  \alpha,\beta,k\right)  }$ for $n=20, \alpha=\beta=1$
fixed and $k=0.3,\ 0.6,\ 0.9,\ 1.2,\ 1.5$ }%
\end{figure}

\section{Bivariate Generalization of the Lupa\c{s}-Kantorovich-Stancu type
operators by means of Pochhammer $k$-symbol}

In the year 2016, Agrawal et al. \cite{AIK} introduced bivariate Kantorovich
variant of the operators in view of the operators given by Lupa\c{s} and
Lupa\c{s} \cite{Lupas} and they investigated the rate of convergence by means
of the modulus of continuity and proved the Voronovskaja type asymptotic
theorem for the bivariate Lupa\c{s}-Kantorovich operators. Also in the same
year, Agrawal et al. \cite{Agrawal1} constructed the bivariate
Lupa\c{s}-Durrmeyer operators and they examined some approximation properties
of these operators using Peetre's K-functional and discussed the asymptotic
behaviour of the operators. Moreover, they considered the Generalized Boolean
Sum operators of Lupa\c{s}--Durrmeyer type operators and the rate of
convergence theorems for the operators were given in that study. In 2018,
inspired by the concept of \"{O}zarslan and Duman \cite{Ozars}, Kajla and
Miclaus \cite{Kajla} generalized the Lupa\c{s}-Kantorovich type operators
based on Polya distribution. They gave the degree of approximation and
discussed the Voronovskaja type theorem for the bivariate operators.
Subsequently, Agrawal and Gupta \cite{Agrawal} introduced the Kantorovich
variant of $q$-analogue of the Stancu operators defined by Nowak \cite{Nowak}
and they presented bivariate operators in their work. Recently, Rahman et al.
\cite{RMK} has presented the bivariate extension of Kantorovich variant of
Lupa\c{s} operators based on Polya distribution with shifted knots $\alpha
_{i},\  \beta_{i},\ i=1,2$. They proposed some aproximation properties of these
operators by means of the modulus of continuity, Peetre's $K$-functional and
provided some illustrative graphics and numerical examples.

In this section, taking into account Polya distribution, we present bivariate
Lupa\c{s} Kantorovich-Stancu type operators via Pochhammer $k$-symbol mainly
motivated by the work \cite{AIK}. Also we give rate of convergence via modulus
of continuity and the Lipschitz class functions for the bivariate case and we
prove the Voronovskaja type theorem for the modified operators. Finally we try
to illustrate the approximation process with some graphics. Let $J^{2}=J\times
J$ where $J=\left[  0,1\right]  $. $C\left(  J^{2}\right)  \ $be the space of
all real-valued continuous functions on $J^{2}\ $endowed with the norm%
\[
\left \Vert g\right \Vert _{C\left(  J^{2}\right)  }=\sup_{x,y\in J}\left \vert
g\left(  x,y\right)  \right \vert .
\]
Let $C^{2}\left(  J^{2}\right)  \ $be the space of all continuous functions $g\in
C\left(  J^{2}\right)  $ such that $g_{x},~g_{y},~g_{xx},~g_{yy}$ belong to
$C\left(  J^{2}\right)  .$ The norm of $g\in C^{2}\left(  J^{2}\right)  $ is
defined by%
\[
\left \Vert g\right \Vert _{C^{2}\left(  J^{2}\right)  }=\left \Vert g\right \Vert
_{C\left(  J^{2}\right)  }+%
{\displaystyle \sum \limits_{i=1}^{2}}
\left(  \left \Vert \frac{\partial^{i}g}{\partial x^{i}}\right \Vert _{C\left(
J^{2}\right)  }+\left \Vert \frac{\partial^{i}g}{\partial y^{i}}\right \Vert
_{C\left(  J^{2}\right)  }\right)  .
\]

We construct Lupa\c{s}-Kantorovich-Stancu type operators in the bivariate case
as follows%
\begin{align}
K_{n_{1},n_{2}}^{\left(  \alpha_{1},\alpha_{2},\beta_{1},\beta_{2},k_{1}%
,k_{2}\right)  }\left(  f;x,y\right)   &  =\frac{\left(  n_{1}+\beta
_{1}+1\right)  \left(  n_{2}+\beta_{2}+1\right)  }{\left(  n_{1}\right)
_{n_{1},k_{1}}\left(  n_{2}\right)  _{n_{2},k_{2}}}\label{BivPol}\\
&  \times \sum_{m_{1}=0}^{n_{1}}\sum_{m_{2}=0}^{n_{2}}\left(
\begin{array}
[c]{c}%
n_{1}\\
m_{1}%
\end{array}
\right)  \left(
\begin{array}
[c]{c}%
n_{2}\\
m_{2}%
\end{array}
\right)  \left(  n_{1}x\right)  _{m_{1},k_{1}}\left(  n_{1}-n_{1}x\right)
_{n_{1}-m_{1},k_{1}}\nonumber \\
&  \times \left(  n_{2}y\right)  _{m_{2},k_{2}}\left(  n_{2}-n_{2}y\right)
_{n_{2}-m_{2},k_{2}}%
{\displaystyle \int \limits_{\frac{m_{1}+\alpha_{1}}{n_{1}+\beta_{1}+1}}%
^{\frac{m_{1}+\alpha_{1}+1}{n_{1}+\beta_{1}+1}}}
{\displaystyle \int \limits_{\frac{m_{2}+\alpha_{2}}{n_{2}+\beta_{2}+1}}%
^{\frac{m_{2}+\alpha_{2}+1}{n_{2}+\beta_{2}+1}}}
f\left(  t,s\right)  dtds,\nonumber
\end{align}
where $n_{1},n_{2}\in$ $%
\mathbb{N}
$ and $k_{1},k_{2},\alpha_{1},\alpha_{2},\beta_{1},\beta_{2}$ are nonzero real
numbers provided $0\leq \alpha_{1}\leq \alpha_{2}\leq \beta_{1}\leq \beta_{2}.$

Now, we establish the following Lemmas. It is obvious that the extension
$K_{n_{1},n_{2}}^{\left(  \mathbf{\alpha},\mathbf{\beta},\mathbf{k}\right)  }$
given above is coincide with the operators in \cite{AIK} when $k_{i}=1\ $and
$\alpha_{i}=0,\  \beta_{i}=0\ $for $i=1,2.$ Also in \cite{RMK}\ taking
$\alpha_{2}=\beta_{2}=0$, the operators in \cite{RMK} reduce to the
$K_{n_{1},n_{2}}^{\left(  \mathbf{\alpha},\mathbf{\beta},\mathbf{1}\right)
}.$

In the following, for the simplicity we use the notation $K_{n_{1},n_{2}%
}^{\left(  \mathbf{\alpha},\mathbf{\beta},\mathbf{k}\right)  }$ instead of
$K_{n_{1},n_{2}}^{\left(  \alpha_{1},\alpha_{2},\beta_{1},\beta_{2}%
,k_{1},k_{2}\right)  }$ which is given as
\[
K_{n_{1},n_{2}}^{\left(  \alpha_{1},\alpha_{2},\beta_{1},\beta_{2},k_{1}%
,k_{2}\right)  }\left(  f;x,y\right)  :=K_{n_{1},n_{2}}^{\left(
\mathbf{\alpha,\beta},\mathbf{k}\right)  }\left(  f;x,y\right)  ,
\]
where $\mathbf{\alpha}=\left(  \alpha_{1},\alpha_{2}\right)  ,\  \mathbf{\beta}=\left(  \beta_{1},\beta_{2}\right)  ,\  \mathbf{k}=\left(  k_{1}%
,k_{2}\right)  .$

First, we state some basic lemmas which are required to prove approximation
properties of $K_{n_{1},n_{2}}^{\left(  \mathbf{\alpha},\mathbf{\beta
},\mathbf{k}\right)  }.\ $Let us denote the monomials $e_{i,j}\left(
x,y\right)  =x^{i}y^{j}$ for $\left(  i,j\right)  \in%
\mathbb{N}
_{0}\times%
\mathbb{N}
_{0}.$

\begin{lemma}
\label{test bivariate} Let $n_{1},n_{2}\in%
\mathbb{N}
$ and $\ k_{1},k_{2}$ be nonnegative real numbers. From Lemma \ref{testkant}, for the
operators $K_{n_{1},n_{2}}^{\left(  \mathbf{\alpha},\mathbf{\beta}%
,\mathbf{k}\right)  }$ defined by $\left(  \text{\ref{BivPol}}\right)  $, we
have%
\begin{align}
K_{n_{1},n_{2}}^{\left(  \mathbf{\alpha},\mathbf{\beta},\mathbf{k}\right)
}\left(  e_{00};x,y\right)   &  =1,\label{bivtest1}\\
K_{n_{1},n_{2}}^{\left(  \mathbf{\alpha},\mathbf{\beta},\mathbf{k}\right)
}\left(  e_{10};x,y\right)   &  =\frac{2\alpha_{1}+1}{2\left(  n_{1}+\beta
_{1}+1\right)  }+\frac{n_{1}x}{n_{1}+\beta_{1}+1},\label{bivtest2}\\
K_{n_{1},n_{2}}^{\left(  \mathbf{\alpha},\mathbf{\beta},\mathbf{k}\right)
}\left(  e_{01};x,y\right)   &  =\frac{2\alpha_{2}+1}{2\left(  n_{2}+\beta
_{2}+1\right)  }+\frac{n_{2}y}{n_{2}+\beta_{2}+1},\label{bivtest3}\\
K_{n_{1},n_{2}}^{\left(  \mathbf{\alpha},\mathbf{\beta},\mathbf{k}\right)
}\left(  e_{20};x,y\right)   &  =\frac{n_{1}^{2}}{\left(  n_{1}+\beta
_{1}+1\right)  ^{2}}\left[  x^{2}+x\left(  1-x\right)  \left(  \frac{k_{1}%
+1}{n_{1}+k_{1}}\right)  \right] \nonumber \\
&  +\frac{\left(  2\alpha_{1}+1\right)  n_{1}x}{\left(  n_{1}+\beta
_{1}+1\right)  ^{2}}+\frac{\left(  \alpha_{1}+1\right)  ^{3}-\alpha_{1}^{3}%
}{3\left(  n_{1}+\beta_{1}+1\right)  ^{2}},\label{bivtest4}\\
K_{n_{1},n_{2}}^{\left(  \mathbf{\alpha},\mathbf{\beta},\mathbf{k}\right)
}\left(  e_{02};x,y\right)   &  =\frac{n_{2}^{2}}{\left(  n_{2}+\beta
_{2}+1\right)  ^{2}}\left[  y^{2}+y\left(  1-y\right)  \left(  \frac{k_{2}%
+1}{n_{2}+k_{2}}\right)  \right] \nonumber \\
&  +\frac{\left(  2\alpha_{2}+1\right)  n_{2}y}{\left(  n_{2}+\beta
_{2}+1\right)  ^{2}}+\frac{\left(  \alpha_{2}+1\right)  ^{3}-\alpha_{2}^{3}%
}{3\left(  n_{2}+\beta_{2}+1\right)  ^{2}}. \label{bivtest5}%
\end{align}

\end{lemma}

\begin{corollary}
\label{centkantbiv} Let $n_{1},n_{2}\in%
\mathbb{N}
$,$\ k_{1},k_{2}\geq0.$ From Lemma \ref{test bivariate}, the central moments of the operators
$K_{n_{1},n_{2}}^{\left(  \mathbf{\alpha},\mathbf{\beta},\mathbf{k}\right)  }$
are given by%
\begin{align}
K_{n_{1},n_{2}}^{\left(  \mathbf{\alpha},\mathbf{\beta},\mathbf{k}\right)
}\left(  e_{10}-x;x,y\right)   &  =\frac{2\alpha_{1}+1}{2\left(  n_{1}+\beta
_{1}+1\right)  }-\frac{\left(  \beta_{1}+1\right)  x}{n_{1}+\beta_{1}%
+1},\label{bivcent1}\\
K_{n_{1},n_{2}}^{\left(  \mathbf{\alpha},\mathbf{\beta},\mathbf{k}\right)
}\left(  e_{01}-y;x,y\right)   &  =\frac{2\alpha_{2}+1}{2\left(  n_{2}+\beta
_{2}+1\right)  }-\frac{\left(  \beta_{2}+1\right)  y}{n_{2}+\beta_{2}%
+1},\label{bivcent2}\\
K_{n_{1},n_{2}}^{\left(  \mathbf{\alpha},\mathbf{\beta},\mathbf{k}\right)
}\left(  \left(  e_{10}-x\right)  ^{2};x,y\right)   &  =\frac{1}{\left(
n_{1}+\beta_{1}+1\right)  ^{2}}\left \{  x\left(  1-x\right)  \left(
\frac{n_{1}^{2}\left(  k_{1}+1\right)  }{n_{1}+k_{1}}-\left(  \beta
_{1}+1\right)  ^{2}\right)  \right. \nonumber \\
&  +\left.  \left(  \beta_{1}+1\right)  \left(  \beta_{1}-2\alpha_{1}\right)
x+\frac{\left(  \alpha_{1}+1\right)  ^{3}-\alpha_{1}^{3}}{3}\right \}
,\label{bivcent3}\\
K_{n_{1},n_{2}}^{\left(  \mathbf{\alpha},\mathbf{\beta},\mathbf{k}\right)
}\left(  \left(  e_{01}-y\right)  ^{2};x,y\right)   &  =\frac{1}{\left(
n_{2}+\beta_{2}+1\right)  ^{2}}\left \{  y\left(  1-y\right)  \left(
\frac{n_{2}^{2}\left(  k_{2}+1\right)  }{n_{2}+k_{2}}-\left(  \beta
_{2}+1\right)  ^{2}\right)  \right. \nonumber \\
&  +\left.  \left(  \beta_{2}+1\right)  \left(  \beta_{2}-2\alpha_{2}\right)
y+\frac{\left(  \alpha_{2}+1\right)  ^{3}-\alpha_{2}^{3}}{3}\right \}  .
\label{bivcent4}%
\end{align}

\end{corollary}

\begin{lemma}
\label{bivkor1} Let $n_{1},n_{2}\in%
\mathbb{N}
$,$\ k_{1},k_{2}\geq0.\ $Then for every $f\in C\left(  J^{2}\right)  ,\ $
\[
\lim_{n_{1},n_{2}\rightarrow \infty}\left(  e_{ij};x,y\right)  =e_{ij},
\]
for $\left(  i,j\right)  \in \left \{  \left(  0,0\right)  ,\left(  0,1\right)
,\left(  1,0\right)  \right \}  $ and
\[
\lim_{n_{1},n_{2}\rightarrow \infty}K_{n_{1},n_{2}}^{\left(  \mathbf{\alpha
},\mathbf{\beta},\mathbf{k}\right)  }\left(  e_{20}+e_{02};x,y\right)
=e_{20}+e_{02},
\]
uniformly on $J^{2},\ J=\left[  0,1\right]  .$
\end{lemma}

\begin{theorem}
\label{w} Let $n_{1},n_{2}\in%
\mathbb{N}
$,$\ k_{1},k_{2}\geq0.\ $Then for every $f\in C\left(  J^{2}\right)  ,$ we
have%
\begin{equation}
\lim_{n_{1},n_{2}\rightarrow \infty}\left \Vert K_{n_{1},n_{2}}^{\left(
\mathbf{\alpha},\mathbf{\beta},\mathbf{k}\right)  }\left(  f\right)
-f\right \Vert =0. \label{bivkor}%
\end{equation}

\end{theorem}

\begin{proof}
According to the Korovkin theorem for the bivariate case given in
\cite{Volko}, by applying the results given in Lemma \ref{bivkor1}, we get the
desired result.
\end{proof}

For any function $f\in C\left(  J^{2}\right)  ,$ the complete modulus of
continuity for bivariate case is defined as follows:%
\begin{equation}
\widetilde{\omega}\left(  f;\delta \right)  =\sup \left \{  \left \vert f\left(
t,s\right)  -f\left(  x,y\right)  \right \vert :\left(  t,s\right)  ,~\left(
x,y\right)  \in J^{2}\text{ and }\sqrt{\left(  t-x\right)  ^{2}+\left(
s-y\right)  ^{2}}\leq \delta \right \}  \label{compmod}%
\end{equation}
Moreover, the partial moduli of continuity of $f$ with respect to $x$ and
$y\ $ is given by$\ $%
\begin{align}
\omega_{1}\left(  f;\delta \right)   &  =\sup \left \{  \left \vert f\left(
x_{1},y\right)  -f\left(  x_{2},y\right)  \right \vert :y\in J\text{ and
}\left \vert x_{1}-x_{2}\right \vert \leq \delta \right \}  ,\label{part1}\\
\omega_{2}\left(  f;\delta \right)   &  =\sup \left \{  \left \vert f\left(
x,y_{1}\right)  -f\left(  x,y_{2}\right)  \right \vert :x\in J\text{ and
}\left \vert y_{1}-y_{2}\right \vert \leq \delta \right \}  . \label{part2}%
\end{align}

\begin{theorem}
\label{p} Let $n_{1},n_{2}\in%
\mathbb{N}
$,$\ k_{1},k_{2}\geq0.\ $Then for every $f\in C\left(  J^{2}\right)  $ for all
$\left(  x,y\right)  \in J^{2},$ we have the following result%
\[
\left \vert K_{n_{1},n_{2}}^{\left(  \mathbf{\alpha},\mathbf{\beta}%
,\mathbf{k}\right)  }\left(  f;x,y\right)  -f\left(  x,y\right)  \right \vert
\leq2\widetilde{\omega}\left(  f;\sqrt{\xi_{n_{1},k_{1}}^{\left(  \alpha
_{1},\beta_{1}\right)  }+\xi_{n_{2},k_{2}}^{\left(  \alpha_{2},\beta
_{2}\right)  }}\right)  ,
\]
where
\begin{equation}
\xi_{n_{i},k_{i}}^{\left(  \alpha_{i},\beta_{i}\right)  }=\frac{1}{n_{i}%
+\beta_{i}+1}\left(  \frac{k_{i}+1}{4}+\beta_{i}+2\alpha_{i}+\frac{\left(
\alpha_{i}+1\right)  ^{3}-\alpha_{i}^{3}}{3}\right)  ,\ i=1,2 \label{**}%
\end{equation}
and $\widetilde{\omega}\left(  f;.\right)  $ is the complete modulus of
continuity defined by $\left(  \text{\ref{compmod}}\right)  $.
\end{theorem}

\begin{proof}
Taking in view the definition of the operators $\left(  \text{\ref{BivPol}%
}\right)  $ and using the complete modulus of continuity $\left(
\text{\ref{compmod}}\right)  $, from the Cauchy--Schwarz inequality it follows
that%
\begin{align*}
&  \left \vert K_{n_{1},n_{2}}^{\left(  \mathbf{\alpha},\mathbf{\beta
},\mathbf{k}\right)  }\left(  f;x,y\right)  -f\left(  x,y\right)  \right \vert
\\
&  \leq K_{n_{1},n_{2}}^{\left(  \mathbf{\alpha},\mathbf{\beta},\mathbf{k}%
\right)  }\left(  \left \vert f\left(  t,s\right)  -f\left(  x,y\right)
\right \vert ;x,y\right) \\
&  \leq \widetilde{\omega}\left(  f;\delta_{n_{1},n_{2}}\right)  \left(
1+\frac{1}{\delta_{n_{1},n_{2}}}K_{n_{1},n_{2}}^{\left(  \mathbf{\alpha
},\mathbf{\beta},\mathbf{k}\right)  }\left(  \sqrt{\left(  t-x\right)
^{2}+\left(  s-y\right)  ^{2}};x,y\right)  \right) \\
&  \leq \widetilde{\omega}\left(  f;\delta_{n_{1},n_{2}}\right)  \left(
1+\frac{1}{\delta_{n_{1},n_{2}}}\left(  K_{n_{1},n_{2}}^{\left(
\mathbf{\alpha},\mathbf{\beta},\mathbf{k}\right)  }\left(  \left(  t-x\right)
^{2}+\left(  s-y\right)  ^{2};x,y\right)  \right)  ^{1/2}\right) \\
&  \leq \widetilde{\omega}\left(  f;\delta_{n_{1},n_{2}}\right)  \left(
1+\frac{1}{\delta_{n_{1},n_{2}}}\left(  K_{n_{1}}^{\left(  \alpha_{1}%
,\beta_{1},k_{1}\right)  }\left(  \left(  t-x\right)  ^{2};x\right)
+K_{n_{2}}^{\left(  \alpha_{2},\beta_{2},k_{2}\right)  }\left(  \left(
s-y\right)  ^{2};y\right)  \right)  ^{1/2}\right)  .
\end{align*}
From Lemma \ref{m}, we find that%
\begin{align*}
&  \left \vert K_{n_{1},n_{2}}^{\left(  \mathbf{\alpha},\mathbf{\beta
},\mathbf{k}\right)  }\left(  f;x,y\right)  -f\left(  x,y\right)  \right \vert
\\
&  \leq \widetilde{\omega}\left(  f;\delta_{n_{1},n_{2}}\right)  \left(
1+\frac{1}{\delta_{n_{1},n_{2}}}\left(  \xi_{n_{1},k_{1}}^{\left(  \alpha
_{1},\beta_{1}\right)  }+\xi_{n_{2},k_{2}}^{\left(  \alpha_{2},\beta
_{2}\right)  }\right)  ^{1/2}\right)
\end{align*}
where
\begin{align*}
\xi_{n_{1},k_{1}}^{\left(  \alpha_{1},\beta_{1}\right)  }  &  =\frac{1}%
{n_{1}+\beta_{1}+1}\left(  \frac{k_{1}+1}{4}+\beta_{1}+2\alpha_{1}%
+\frac{\left(  \alpha_{1}+1\right)  ^{3}-\alpha_{1}{}^{3}}{3}\right)  ,\\
\xi_{n_{2},k_{2}}^{\left(  \alpha_{2},\beta_{2}\right)  }  &  =\frac{1}%
{n_{2}+\beta_{2}+1}\left(  \frac{k_{2}+1}{4}+\beta_{2}+2\alpha_{2}%
+\frac{\left(  \alpha_{2}+1\right)  ^{3}-\alpha_{2}{}^{3}}{3}\right)  .
\end{align*}
Taking%
\[
\delta_{n_{1},n_{2}}=\sqrt{\xi_{n_{1},k_{1}}^{\left(  \alpha_{1},\beta
_{1}\right)  }+\xi_{n_{2},k_{2}}^{\left(  \alpha_{2},\beta_{2}\right)  }},
\]
we reach the required result.
\end{proof}

\begin{theorem}
Let $n_{1},n_{2}\in%
\mathbb{N}
$,$\ k_{1},k_{2}\geq0.\ $Then for every $f\in C\left(  J^{2}\right)  $ for all
$\left(  x,y\right)  \in J^{2},$ we have the following result%
\[
\left \vert K_{n_{1},n_{2}}^{\left(  \mathbf{\alpha},\mathbf{\beta}%
,\mathbf{k}\right)  }\left(  f;x,y\right)  -f\left(  x,y\right)  \right \vert
\leq2\left(  \omega^{\left(  1\right)  }\left(  f;\sqrt{\xi_{n_{1},k_{1}%
}^{\left(  \alpha_{1},\beta_{1}\right)  }}\right)  +\omega^{\left(  2\right)
}\left(  f;\sqrt{\xi_{n_{2},k_{2}}^{\left(  \alpha_{2},\beta_{2}\right)  }%
}\right)  \right)  ,
\]
where $\omega_{1}\left(  f;.\right)  $ and $\omega_{2}\left(  f;.\right)
\ $are the partial moduli of continuity of $f$ defined by $\left(
\text{\ref{part1}}\right)  $ and $\left(  \text{\ref{part2}}\right)  ,$
respectively and $\xi_{n_{i},k_{i}}^{\left(  \alpha_{i},\beta_{i}\right)  }$
is as in Theorem \ref{p}.
\end{theorem}

\begin{proof}
Directly from $\left(  \text{\ref{BivPol}}\right)  $ and the using the Cauchy
Schwarz inequality, we easily obtain%
\[%
\begin{array}
[c]{l}%
\left \vert K_{n_{1},n_{2}}^{\left(  \mathbf{\alpha},\mathbf{\beta}%
,\mathbf{k}\right)  }\left(  f;x,y\right)  -f\left(  x,y\right)  \right \vert
\leq K_{n_{1},n_{2}}^{\left(  \mathbf{\alpha},\mathbf{\beta},\mathbf{k}%
\right)  }\left(  \left \vert f\left(  t,s\right)  -f\left(  x,y\right)
\right \vert ;x,y\right) \\
\\
\leq K_{n_{1},n_{2}}^{\left(  \mathbf{\alpha},\mathbf{\beta},\mathbf{k}%
\right)  }\left(  \left \vert f\left(  t,s\right)  -f\left(  t,y\right)
\right \vert ;x,y\right)  +K_{n_{1},n_{2}}^{\left(  \mathbf{\alpha
},\mathbf{\beta},\mathbf{k}\right)  }\left(  \left \vert f\left(  t,y\right)
-f\left(  x,y\right)  \right \vert ;x,y\right) \\
\\
\leq K_{n_{1},n_{2}}^{\left(  \mathbf{\alpha},\mathbf{\beta},\mathbf{k}%
\right)  }\left(  \omega^{\left(  1\right)  }\left(  f;\left \vert
t-x\right \vert \right)  ;x,y\right)  +K_{n_{1},n_{2}}^{\left(  \mathbf{\alpha
},\mathbf{\beta},\mathbf{k}\right)  }\left(  \omega^{\left(  2\right)
}\left(  f;\left \vert s-y\right \vert \right)  ;x,y\right) \\
\\
\leq \omega^{\left(  1\right)  }\left(  f;\delta_{n_{1}}\right)  \left(
1+\frac{1}{\delta_{n_{1}}}K_{n_{1}}^{\left(  \alpha_{1},\beta_{1}%
,k_{1}\right)  }\left(  \left \vert t-x\right \vert ;x\right)  \right) \\
\\
+\omega^{\left(  2\right)  }\left(  f;\delta_{n_{2}}\right)  \left(
1+\frac{1}{\delta_{n_{2}}}K_{n_{2}}^{\left(  \alpha_{2},\beta_{2}%
,k_{2}\right)  }\left(  \left \vert s-y\right \vert ;y\right)  \right) \\
\\
\leq \omega^{\left(  1\right)  }\left(  f;\delta_{n_{1}}\right)  \left(
1+\frac{1}{\delta_{n_{1}}}\left(  K_{n_{1}}^{\left(  \alpha_{1},\beta
_{1},k_{1}\right)  }\left(  \left(  t-x\right)  ^{2};x\right)  \right)
^{1/2}\right) \\
\\
+\omega^{\left(  2\right)  }\left(  f;\delta_{n_{2}}\right)  \left(
1+\frac{1}{\delta_{n_{2}}}\left(  K_{n_{2}}^{\left(  \alpha_{2},\beta
_{2},k_{2}\right)  }\left(  \left(  s-y\right)  ^{2};y\right)  \right)
^{1/2}\right)
\end{array}
\]
From Lemma \ref{m},
\begin{align*}
\left \vert K_{n_{1},n_{2}}^{\left(  \mathbf{\alpha},\mathbf{\beta}%
,\mathbf{k}\right)  }\left(  f;x,y\right)  -f\left(  x,y\right)  \right \vert
&  \leq \omega^{\left(  1\right)  }\left(  f;\delta_{n_{1}}\right)  \left(
1+\frac{1}{\delta_{n_{1}}}\left(  \xi_{n_{1},k_{1}}^{\left(  \alpha_{1}%
,\beta_{1}\right)  }\right)  ^{1/2}\right) \\
&  +\omega^{\left(  2\right)  }\left(  f;\delta_{n_{2}}\right)  \left(
1+\frac{1}{\delta_{n_{2}}}\left(  \xi_{n_{2},k_{2}}^{\left(  \alpha_{2}%
,\beta_{2}\right)  }\right)  ^{1/2}\right)
\end{align*}
Taking $\delta_{n_{1}}=\sqrt{\xi_{n_{1},k_{1}}^{\left(  \alpha_{1},\beta
_{1}\right)  }}$ and $\delta_{n_{2}}=\sqrt{\xi_{n_{2},k_{2}}^{\left(
\alpha_{2},\beta_{2}\right)  }}$, we complete the proof.
\end{proof}

We continue by recalling the definition of the Lipschitz class for bivariate
function of $f.\ $It is known that a function $f$ belongs to $Lip_{M}\left(
\gamma_{1},\gamma_{2}\right)  $ if it satisfies
\[
\left \vert f\left(  t,s\right)  -f\left(  x,y\right)  \right \vert \leq
M\left \vert t-x\right \vert ^{\gamma_{1}}\left \vert s-y\right \vert ^{\gamma
_{2}},
\]
where $M>0,\ 0<\gamma_{1}\leq1,\ $ $0<\gamma_{2}\leq1.\ $We are in a position
to prove of the rate of convergence for the bivariate operators by virtue of
the Lipschitz class.

\begin{theorem}
\label{x} Let $M>0,\ 0<\gamma_{1}\leq1,\ 0<\gamma_{2}\leq1\ $and $f\in Lip_{M}\left(
\gamma_{1},\gamma_{2}\right)  .\ $Then%
\[
\left \Vert K_{n_{1},n_{2}}^{\left(  \mathbf{\alpha},\mathbf{\beta}%
,\mathbf{k}\right)  }\left(  f\right)  -f\right \Vert \leq M\left(
\lambda_{n_{1},k_{2}}^{\left(  \alpha_{1},\beta_{1}\right)  }\right)
^{\frac{\gamma_{1}}{2}}\left(  \lambda_{n_{2},k_{2}}^{\left(  \alpha_{2}%
,\beta_{2}\right)  }\right)  ^{\frac{\gamma_{2}}{2}},
\]
where $\lambda_{n_{1},k_{1}}^{\left(  \alpha_{1},\beta_{1}\right)
}=\left \Vert K_{n_{1}}^{\left(  \alpha_{1},\beta_{1},k_{1}\right)  }\left(
\left(  t-.\right)  ^{2};.\right)  \right \Vert $and $\lambda_{n_{2},k_{2}%
}^{\left(  \alpha_{2},\beta_{2}\right)  }=\left \Vert K_{n_{2}}^{\left(
\alpha_{2},\beta_{2},k_{2}\right)  }\left(  \left(  s-.\right)  ^{2};.\right)
\right \Vert .$
\end{theorem}

\begin{proof}
Taking into account $f\in Lip_{M}\left(  \gamma_{1},\gamma_{2}\right)  $, this
allows us to write%
\begin{align*}
\left \vert K_{n_{1},n_{2}}^{\left(  \mathbf{\alpha},\mathbf{\beta}%
,\mathbf{k}\right)  }\left(  f;x,y\right)  -f\left(  x,y\right)  \right \vert
&  \leq K_{n_{1},n_{2}}^{\left(  \mathbf{\alpha},\mathbf{\beta},\mathbf{k}%
\right)  }\left(  \left \vert f\left(  t,s\right)  -f\left(  x,y\right)
\right \vert ;x,y\right) \\
&  \leq MK_{n_{1},n_{2}}^{\left(  \mathbf{\alpha},\mathbf{\beta}%
,\mathbf{k}\right)  }\left(  \left \vert t-x\right \vert ^{\gamma_{1}}\left \vert
s-y\right \vert ^{\gamma_{2}};x,y\right) \\
&  =MK_{n_{1}}^{\left(  \alpha_{1},\beta_{1},k_{1}\right)  }\left(  \left \vert
t-x\right \vert ^{\gamma_{1}};x\right)  K_{n_{2}}^{\left(  \alpha_{2},\beta
_{2},k_{2}\right)  }\left(  \left \vert s-y\right \vert ^{\gamma_{2}};y\right)
.
\end{align*}
Applying the H\"{o}lder's inequality, we easily get
\begin{align*}
\left \vert K_{n_{1},n_{2}}^{\left(  \mathbf{\alpha},\mathbf{\beta}%
,\mathbf{k}\right)  }\left(  f;x,y\right)  -f\left(  x,y\right)  \right \vert
&  \leq M\left(  K_{n_{1}}^{\left(  \alpha_{1},\beta_{1},k_{1}\right)
}\left(  \left(  t-x\right)  ^{2};x\right)  \right)  ^{\gamma_{1}/2}\left(
K_{n_{1}}^{\left(  \alpha_{1},\beta_{1},k_{1}\right)  }\left(  e_{0};x\right)
\right)  ^{\frac{2-\gamma_{1}}{2}}\\
&  \times \left(  K_{n_{2}}^{\left(  \alpha_{2},\beta_{2},k_{2}\right)
}\left(  \left(  s-y\right)  ^{2};y\right)  \right)  ^{\gamma_{2}/2}\left(
K_{n_{2}}^{\left(  \alpha_{2},\beta_{2},k_{2}\right)  }\left(  e_{0};y\right)
\right)  ^{\frac{2-\gamma_{2}}{2}}\\
&  \leq M\left(  \lambda_{n_{1},k_{1}}^{\left(  \alpha_{1},\beta_{1}\right)
}\right)  ^{\frac{\gamma_{1}}{2}}\left(  \lambda_{n_{2},k_{2}}^{\left(
\alpha_{2},\beta_{2}\right)  }\right)  ^{\frac{\gamma_{2}}{2}}.
\end{align*}

\end{proof}

\begin{theorem}
Let $n_{1},n_{2}\in%
\mathbb{N}
$,$\ k_{1},k_{2}\geq0.\ $If $f$ has partial derivatives $f_{x}$ and $f_{y}$ on
$J^{2},$ then for all $\left(  x,y\right)  \in J^{2},\ $we have%
\[
\left \Vert K_{n_{1},n_{2}}^{\left(  \mathbf{\alpha},\mathbf{\beta}%
,\mathbf{k}\right)  }\left(  f\right)  -f\right \Vert \leq \left \Vert
f_{x}\right \Vert _{C\left(  J^{2}\right)  }\left(  \lambda_{n_{1},k_{1}%
}^{\left(  \alpha_{1},\beta_{1}\right)  }\right)  ^{\frac{1}{2}}+\left \Vert
f_{y}\right \Vert _{C\left(  J^{2}\right)  }\left(  \lambda_{n_{2},k_{2}%
}^{\left(  \alpha_{2},\beta_{2}\right)  }\right)  ^{\frac{1}{2}},
\]
where $\lambda_{n_{1},k_{1}}^{\left(  \alpha_{1},\beta_{1}\right)  }$and
$\lambda_{n_{2},k_{2}}^{\left(  \alpha_{2},\beta_{2}\right)  }\ $are given as
in Theorem \ref{x}.
\end{theorem}

\begin{proof}
Since $f\in C^{1}\left(  J^{2}\right)  ,$ we can write as follows%
\begin{equation}
f\left(  t,s\right)  -f\left(  x,y\right)  =%
{\displaystyle \int \limits_{x}^{t}}
f_{u}\left(  u,s\right)  du+%
{\displaystyle \int \limits_{y}^{s}}
f_{v}\left(  x,v\right)  dv.\label{int}%
\end{equation}
Taking $K_{n_{1},n_{2}}^{\left(  \mathbf{\alpha},\mathbf{\beta},\mathbf{k}%
\right)  }$on both sides of $\left(  \text{\ref{int}}\right)  $, we obtain%
\[
\left \vert K_{n_{1},n_{2}}^{\left(  \mathbf{\alpha},\mathbf{\beta}%
,\mathbf{k}\right)  }\left(  f;x,y\right)  -f\left(  x,y\right)  \right \vert
\leq K_{n_{1},n_{2}}^{\left(  \mathbf{\alpha},\mathbf{\beta},\mathbf{k}%
\right)  }\left(  \left \vert
{\displaystyle \int \limits_{x}^{t}}
f_{u}\left(  u,s\right)  du\right \vert ;x,y\right)  +K_{n_{1},n_{2}}^{\left(
\mathbf{\alpha},\mathbf{\beta},\mathbf{k}\right)  }\left(  \left \vert
{\displaystyle \int \limits_{y}^{s}}
f_{v}\left(  x,v\right)  dv\right \vert ;x,y\right)  .
\]
With the help of the inequalities which are given as follows
\[
\left \vert
{\displaystyle \int \limits_{x}^{t}}
f_{u}\left(  u,s\right)  du\right \vert \leq \left \Vert f_{x}\right \Vert
_{C\left(  J^{2}\right)  }\left \vert t-x\right \vert \text{ and }\left \vert
{\displaystyle \int \limits_{y}^{s}}
f_{v}\left(  x,v\right)  dv\right \vert \leq \left \Vert f_{y}\right \Vert
_{C\left(  J^{2}\right)  }\left \vert s-y\right \vert ,
\]
we easily reach%
\[
\left \vert K_{n_{1},n_{2}}^{\left(  \mathbf{\alpha},\mathbf{\beta}%
,\mathbf{k}\right)  }\left(  f;x,y\right)  -f\left(  x,y\right)  \right \vert
\leq \left \Vert f_{x}\right \Vert _{C\left(  J^{2}\right)  }K_{n_{1}}^{\left(
\alpha_{1},\beta_{1},k_{1}\right)  }\left(  \left \vert t-x\right \vert
;x\right)  +\left \Vert f_{y}\right \Vert _{C\left(  J^{2}\right)  }K_{n_{2}%
}^{\left(  \alpha_{2},\beta_{2},k_{2}\right)  }\left(  \left \vert
s-y\right \vert ;y\right)  .
\]
By considering Cauchy Schwarz inequality, from Corollary \ref{centkantbiv}, we can deduce the
desired result as
\begin{align*}
\left \vert K_{n_{1},n_{2}}^{\left(  \mathbf{\alpha},\mathbf{\beta}%
,\mathbf{k}\right)  }\left(  f;x,y\right)  -f\left(  x,y\right)  \right \vert
&  \leq \left \Vert f_{x}\right \Vert _{C\left(  J^{2}\right)  }\left(  K_{n_{1}%
}^{\left(  \alpha_{1},\beta_{1},k_{1}\right)  }\left(  \left(  t-x\right)
^{2};x\right)  \right)  ^{\frac{1}{2}}\left(  K_{n_{1}}^{\left(  \alpha
_{1},\beta_{1},k_{1}\right)  }\left(  e_{0};x\right)  \right)  ^{\frac{1}{2}%
}\\
&  +\left \Vert f_{y}\right \Vert _{C\left(  J^{2}\right)  }\left(  K_{n_{2}%
}^{\left(  \alpha_{2},\beta_{2},k_{2}\right)  }\left(  \left(  s-y\right)
^{2};y\right)  \right)  ^{\frac{1}{2}}\left(  K_{n_{2}}^{\left(  \alpha
_{2},\beta_{2},k_{2}\right)  }\left(  e_{0};y\right)  \right)  ^{\frac{1}{2}%
}\\
&  \leq \left \Vert f_{x}\right \Vert _{C\left(  J^{2}\right)  }\left(
\lambda_{n_{1},k_{1}}^{\left(  \alpha_{1},\beta_{1}\right)  }\right)
^{\frac{1}{2}}+\left \Vert f_{y}\right \Vert _{C\left(  J^{2}\right)  }\left(
\lambda_{n_{2},k_{2}}^{\left(  \alpha_{2},\beta_{2}\right)  }\right)
^{\frac{1}{2}}.
\end{align*}

\end{proof}

\begin{theorem}
\label{z} For $f\in C^{2}\left(  J^{2}\right)  ,\ $then
\begin{align*}
\lim_{n\rightarrow \infty}n\left(  K_{n,n}^{\left(  \mathbf{\alpha
},\mathbf{\beta},\mathbf{k}\right)  }\left(  f;x,y\right)  -f\left(
x,y\right)  \right)   &  =\left(  \alpha_{1}+\frac{1}{2}-\left(  \beta
_{1}+1\right)  x\right)  f_{x}\left(  x,y\right)  +\left(  \alpha_{2}+\frac
{1}{2}-\left(  \beta_{2}+1\right)  y\right)  f_{y}\left(  x,y\right) \\
&  +\frac{1}{2}\left(  k_{1}+1\right)  x\left(  1-x\right)  f_{xx}\left(
x,y\right)  +\frac{1}{2}\left(  k_{2}+1\right)  y\left(  1-y\right)
f_{yy}\left(  x,y\right)
\end{align*}
uniformly on $J^{2}.$
\end{theorem}

\begin{proof}
Let $\left(  x,y\right)  \in J^{2}$ be arbitrary. In view of the Taylor's
series expansion of the function $f$ at the point $\left(  x,y\right)  ,$we
obtain%
\begin{align}
f\left(  t,s\right)   &  =f\left(  x,y\right)  +f_{x}\left(  x,y\right)
\left(  t-x\right)  +f_{y}\left(  x,y\right)  \left(  s-y\right) \nonumber \\
&  +\frac{1}{2}\left(  f_{xx}\left(  x,y)\right.  \left(  t-x\right)
^{2}+2f_{xy}\left(  x,y\right)  \left(  t-x\right)  \left(  s-y\right)
\right. \nonumber \\
&  +\left.  f_{yy}\left(  x,y\right)  \left(  s-y\right)  ^{2}\right)
+\Omega \left(  t,s;x,y\right)  \sqrt{\left(  t-x\right)  ^{4}+\left(
s-y\right)  ^{4}} \label{v}%
\end{align}
for $\left(  t,s\right)  \in J^{2}\ $where $\Omega \left(  t,s;x,y\right)  \in
C\left(  J^{2}\right)  $ and $\underset{\left(  t,s\right)  \rightarrow \left(
x,y\right)  }{\lim}\Omega \left(  t,s;x,y\right)  =0.$

Applying the operators $K_{n,n}^{\left(  \mathbf{\alpha},\mathbf{\beta
},\mathbf{k}\right)  }$ to the both sides of $\left(  \text{\ref{v}}\right)
$, it folows%
\begin{align}
&  \lim_{n\rightarrow \infty}n\left(  K_{n,n}^{\left(  \mathbf{\alpha
},\mathbf{\beta},\mathbf{k}\right)  }\left(  f;x,y\right)  -f\left(
x,y\right)  \right) \nonumber \\
&  =f_{x}\left(  x,y\right)  \lim_{n\rightarrow \infty}nK_{n}^{\left(
\alpha_{1},\beta_{1},k_{1}\right)  }\left(  \left(  t-x\right)  ;x\right)
+f_{y}\left(  x,y\right)  \lim_{n\rightarrow \infty}nK_{n}^{\left(  \alpha
_{2},\beta_{2},k_{2}\right)  }\left(  \left(  s-y\right)  ;y\right)
\nonumber \\
&  +\frac{1}{2}f_{xx}\left(  x,y\right)  \lim_{n\rightarrow \infty}%
nK_{n}^{\left(  \alpha_{1},\beta_{1},k_{1}\right)  }\left(  \left(
t-x\right)  ^{2};x\right)  +f_{xy}\left(  x,y\right)  \lim_{n\rightarrow
\infty}nK_{n,n}^{\left(  \mathbf{\alpha},\mathbf{\beta},\mathbf{k}\right)
}\left(  \left(  t-x\right)  \left(  s-y\right)  ;x,y\right) \nonumber \\
&  +\frac{1}{2}f_{yy}\left(  x,y\right)  \lim_{n\rightarrow \infty}%
nK_{n}^{\left(  \alpha_{2},\beta_{2},k_{2}\right)  }\left(  \left(
s-y\right)  ^{2};y\right) \nonumber \\
&  +\lim_{n\rightarrow \infty}nK_{n,n}^{\left(  \mathbf{\alpha},\mathbf{\beta
},\mathbf{k}\right)  }\left(  \Omega \left(  t,s;x,y\right)  \sqrt{\left(
t-x\right)  ^{4}+\left(  s-y\right)  ^{4}};x,y\right)  . \label{AA}%
\end{align}
Using the H\"{o}lder's inequality to the last term of the right side of the
equation $\left(  \text{\ref{AA}}\right)  $, we reach
\begin{align*}
&  \left \vert K_{n,n}^{\left(  \mathbf{\alpha},\mathbf{\beta},\mathbf{k}%
\right)  }\left(  \Omega \left(  t,s;x,y\right)  \sqrt{\left(  t-x\right)
^{4}+\left(  s-y\right)  ^{4}};x,y\right)  \right \vert \\
&  \leq \left(  K_{n,n}^{\left(  \mathbf{\alpha},\mathbf{\beta},\mathbf{k}%
\right)  }\left(  \Omega^{2}\left(  t,s;x,y\right)  ;x,y\right)  \right)
^{\frac{1}{2}}\left(  K_{n,n}^{\left(  \mathbf{\alpha},\mathbf{\beta
},\mathbf{k}\right)  }\left(  \left(  t-x\right)  ^{4}+\left(  s-y\right)
^{4};x,y\right)  \right)  ^{\frac{1}{2}}\\
&  \leq \left(  K_{n,n}^{\left(  \mathbf{\alpha},\mathbf{\beta},\mathbf{k}%
\right)  }\left(  \Omega^{2}\left(  t,s;x,y\right)  ;x,y\right)  \right)
^{\frac{1}{2}}\left(  K_{n}^{\left(  \alpha_{1},\beta_{1},k_{1}\right)
}\left(  \left(  t-x\right)  ^{4};x\right)  +K_{n}^{\left(  \alpha_{2}%
,\beta_{2},k_{2}\right)  }\left(  \left(  s-y\right)  ^{4};y\right)  \right)
^{\frac{1}{2}}.
\end{align*}
Since $K_{n,n}^{\left(  \mathbf{\alpha},\mathbf{\beta},\mathbf{k}\right)
}\left(  \Omega^{2}\left(  t,s;x,y\right)  ;x,y\right)  \rightarrow0$ as
$n\rightarrow \infty$ uniformly on $J^{2}$ from Theorem \ref{w}, by using limit which
is given by $\left(  \text{\ref{limkant3}}\right)  ,\ $we have%
\[
\lim_{n\rightarrow \infty}nK_{n,n}^{\left(  \mathbf{\alpha},\mathbf{\beta
},\mathbf{k}\right)  }\left(  \Omega \left(  t,s;x,y\right)  \sqrt{\left(
t-x\right)  ^{4}+\left(  s-y\right)  ^{4}};x,y\right)  =0.
\]
In view of Corollary \ref{t}, since
\[
\lim_{n\rightarrow \infty}nK_{n,n}^{\left(  \mathbf{\alpha},\mathbf{\beta
},\mathbf{k}\right)  }\left(  \left(  t-x\right)  \left(  s-y\right)
;x,y\right)  =\lim_{n\rightarrow \infty}nK_{n}^{\left(  \alpha_{1},\beta
_{1},k_{1}\right)  }\left(  t-x;x\right)  K_{n}^{\left(  \alpha_{2},\beta
_{2},k_{2}\right)  }\left(  s-y;y\right)  =0,
\]
we thus find that%
\begin{align*}
&  \lim_{n\rightarrow \infty}n\left(  K_{n,n}^{\left(  \mathbf{\alpha
},\mathbf{\beta},\mathbf{k}\right)  }\left(  f;x,y\right)  -f\left(
x,y\right)  \right) \\
&  =\left(  \alpha_{1}+\frac{1}{2}-\left(  \beta_{1}+1\right)  x\right)
f_{x}\left(  x,y\right)  +\left(  \alpha_{2}+\frac{1}{2}-\left(  \beta
_{2}+1\right)  y\right)  f_{y}\left(  x,y\right) \\
&  +\frac{1}{2}\left(  k_{1}+1\right)  x\left(  1-x\right)  f_{xx}\left(
x,y\right)  +\frac{1}{2}\left(  k_{2}+1\right)  y\left(  1-y\right)
f_{yy}\left(  x,y\right)  ,
\end{align*}
which completes the proof.
\end{proof}

\begin{remark}
The case $\alpha_{1}=\alpha_{2}=\beta_{1}=\beta_{2}=0,~k_{1}=k_{2}=1$ in
Theorem \ref{z} presents the Voronovskaja type theorem given for bivariate
operators $D_{n,n}^{\ast \left(  \frac{1}{n},\frac{1}{n}\right)  }$ by Agrawal
et al. \cite{AIK}.
\end{remark}

\begin{example}
Let $f\left(  x,y\right)  =2x^{2}y\cos \left(  \frac{5\pi x}{2}\right)
,\ n_{1}=n_{2}=10$ and $k_{1}=k_{2}=0.2\ $and $\alpha_{1}=\alpha_{2}=\beta
_{1}=\beta_{2}=0.\ $Convergence for the bivariate generalized operators
$K_{n_{1},n_{2}}^{\left(  \mathbf{\alpha},\mathbf{\beta},\mathbf{k}\right)  }$
(green) and the bivariate operators $D_{n_{1},n_{2}}^{\ast \left(  \frac
{1}{n_{1}},\frac{1}{n_{2}}\right)  }$ (yellow) which was given in \cite{AIK}
$\ $to the function $f$ (red) is demonstrated in Figure 7. It can be noted
that for $k_{1},k_{2}=0.2$, the approximation by the operators $K_{n_{1}%
,n_{2}}^{\left(  \mathbf{\alpha},\mathbf{\beta},\mathbf{k}\right)  }$ is
better than the operators $D_{n_{1},n_{2}}^{\ast \left(  \frac{1}{n_{1}}%
,\frac{1}{n_{2}}\right)  }$ to the function $f$.
\end{example}

\begin{example}
Consider $f\left(  x,y\right)  =2x\cos \left(  3\pi \left(  x+y\right)  \right)
,\ k_{1}=k_{2}=0.4\ $and $\alpha_{1}=\alpha_{2}=\beta_{1}=\beta_{2}=0.$ Figure
8 presents the approximation process of the bivariate operators $K_{n_{1}%
,n_{2}}^{\left(  \mathbf{\alpha},\mathbf{\beta},\mathbf{k}\right)  }$ to the
function $f\ $(red) for $n_{1}=n_{2}=10,\ 20,\ 40\ $(yellow, green, blue,
respectively). It is clearly seen that as the values of $n_{1},\ n_{2}$
increase,\ the approximation of the operators $K_{n_{1},n_{2}}^{\left(
\mathbf{\alpha},\mathbf{\beta},\mathbf{k}\right)  }\ $is getting better.
\end{example}

\begin{example}
Let $f\left(  x,y\right)  =7x^{5}\left(  x-\frac{1}{4}\right)  \sin \left(
2\pi y\right)  \ $(red)$.$ Figure 9 shows the convergence of $K_{n_{1},n_{2}%
}^{\left(  \mathbf{\alpha},\mathbf{\beta},\mathbf{k}\right)  }$ to the
function $f\ $for $n_{1}=n_{2}=10,\  \alpha_{1}=\alpha_{2}=\beta_{1}=\beta
_{2}=0\ $fixed and $k_{1},k_{2}=0.3,\ 0.9,\ 1.2$ (yellow, green, blue,
respectively).\ As the value of $k_{1},\ k_{2}$ decreases towards to
zero,\ the approximation of the bivariate operators $K_{n_{1},n_{2}}^{\left(
\mathbf{\alpha},\mathbf{\beta},\mathbf{k}\right)  }\ $is getting better.
\end{example}

\begin{figure}[pth]
	\includegraphics*[width=7cm]{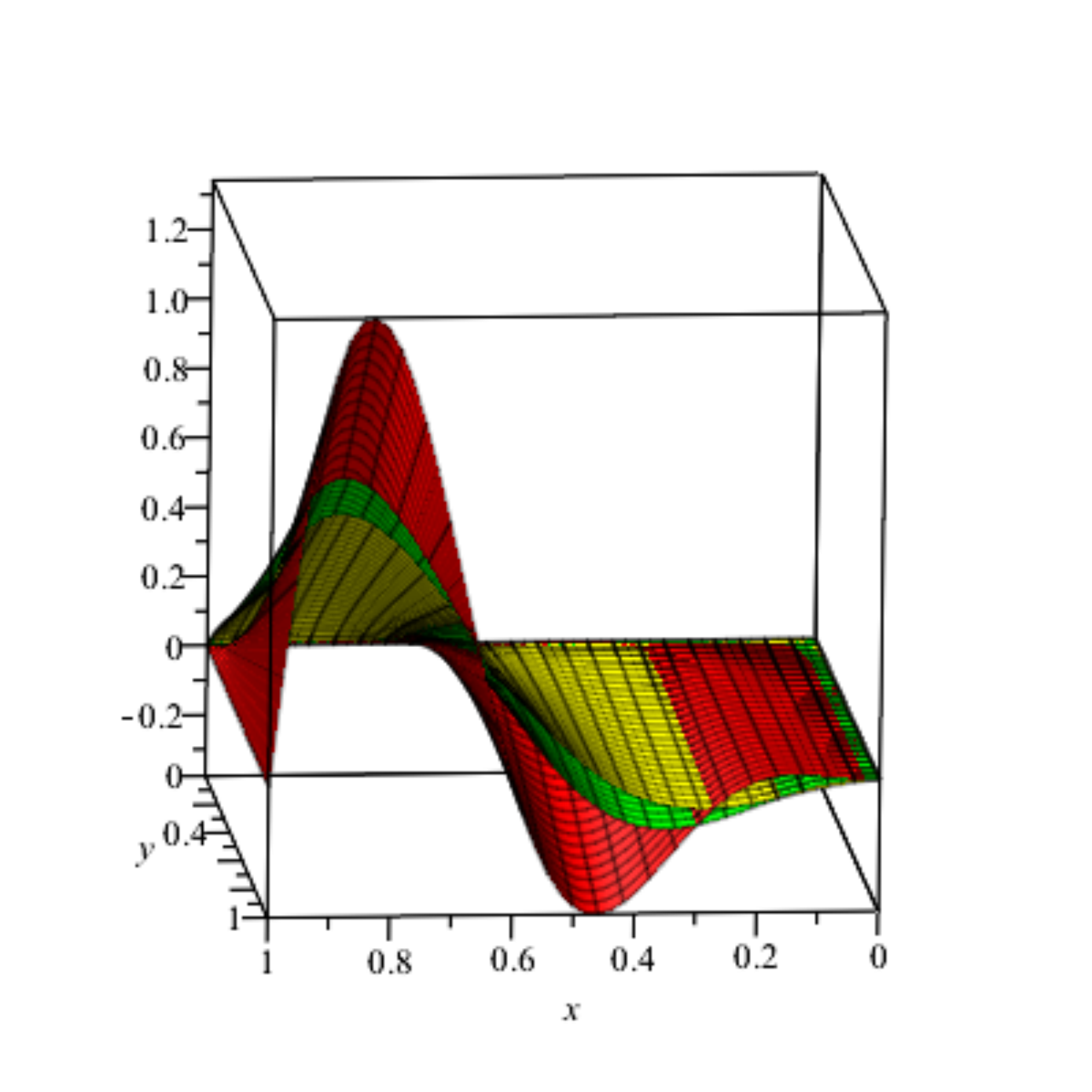}\caption{Convergence of
		${K}_{n_{1},n_{2}}^{\left(  \mathbf{\alpha},\mathbf{\beta
			},\mathbf{k}\right)  }$ and $D_{n_{1},n_{2}}^{\ast \left(  \frac{1}{n_{1}%
			},\frac{1}{n_{2}}\right)  }$ to the function $f$ }%
\end{figure}

\begin{figure}[pth]
	\includegraphics*[width=7cm]{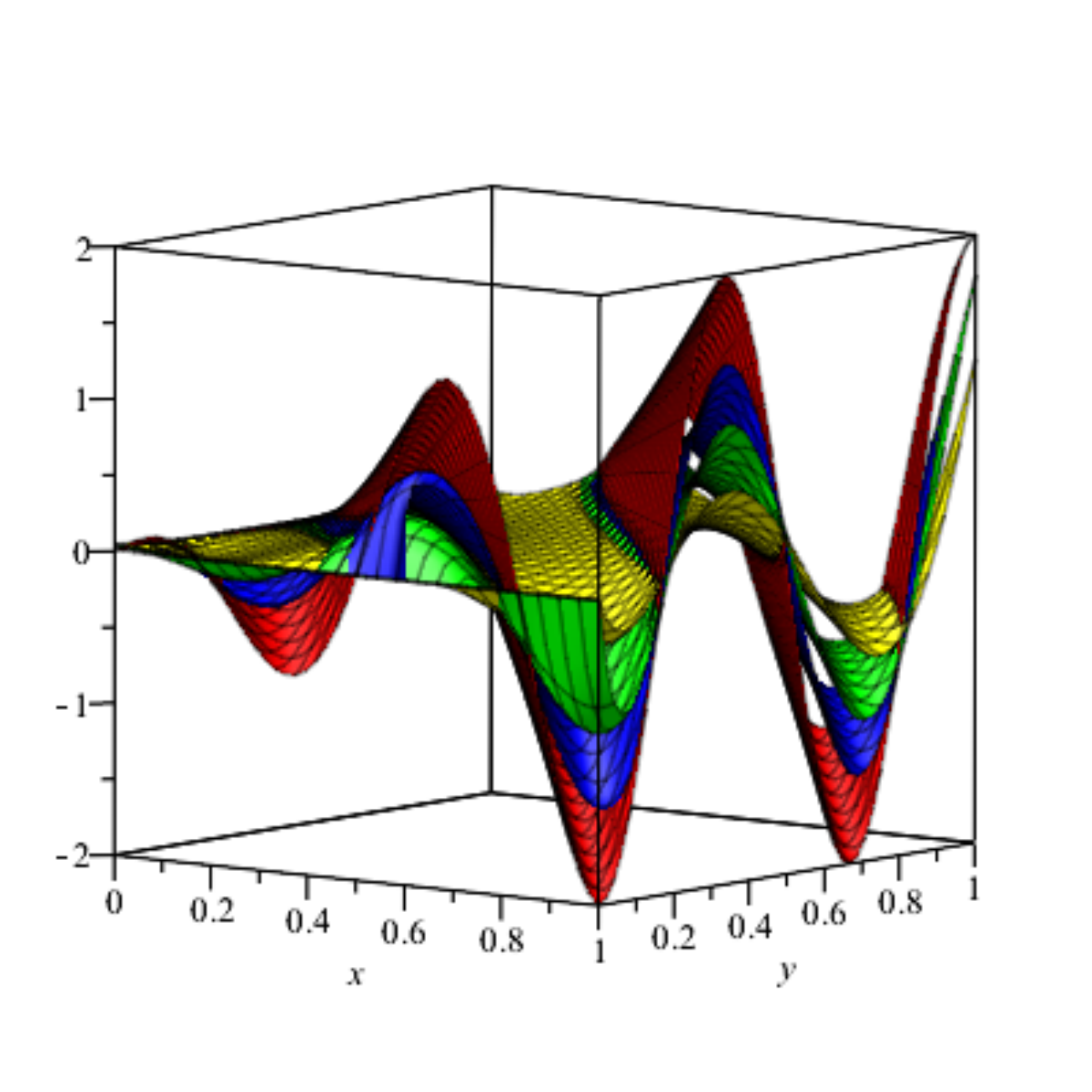}\caption{Approximation of the
		operators ${K}_{n_{1},n_{2}}^{\left(  \mathbf{\alpha},\mathbf{\beta
			},\mathbf{k}\right)  }$, for $n_{1}=n_{2}=10,\ 20,\ 40 $ }%
\end{figure}

\newpage

\begin{figure}[pth]
\includegraphics*[width=7cm]{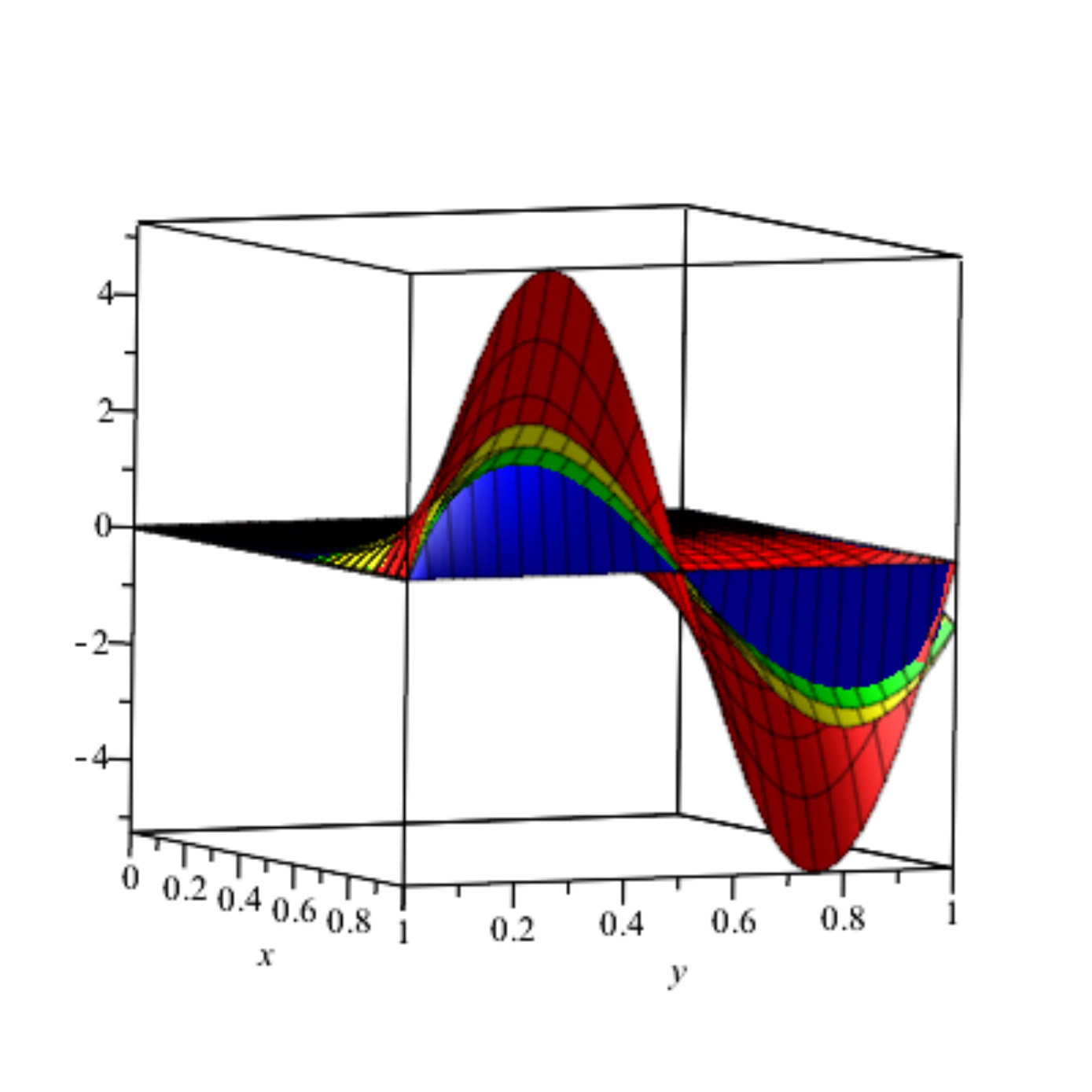}\caption{Approximation of the
operators ${K}_{n_{1},n_{2}}^{\left(  \mathbf{\alpha},\mathbf{\beta
},\mathbf{k}\right)  }$, for $k_{1},k_{2}=0.3,\ 0.9,\ 1.2$ }%
\end{figure}

\section*{Acknowledgements}

This work of the first, second and fourth authors was supported by Scientific
Research Projects Coordination Unit of K\i r\i kkale University. Project
number 2020/045.


\begin{thebibliography}{99}                                                                                               %


\bibitem {Acar}Acar, T., Aral, A., Ra\c{s}a, I., Iterated Boolean Sums of
Bernstein Type Operators, Numerical Functional Analysis and Optimization,
1-13, (2020).

\bibitem {Acu}Acu, A. M., Gonska, H., Perturbed Bernstein-type operators,
Analysis and Mathematical Physics, 10(4),1-26, (2020).

\bibitem {AMS}Acu, A.M., Manav, N., Sofonea, D.F., Approximation properties of
$\lambda$-Kantorovich operators, Journal of Inequalities and Applications,
2018: 202, (2018).

\bibitem {Agrawal}Agrawal, P. N., Gupta, P., $q$-Lupa\c{s} Kantorovich
operators based on P\'{o}lya distribution, Ann. Univ. Ferrara 64, 1--23, (2018).

\bibitem {Agrawal1}Agrawal, P.N., Ispir, N., Kajla, A., GBS Operators of
Lupa\c{s}--Durrmeyer type based on P\'{o}lya Distribution, Results Math.
69(3--4), 397--418, (2016).

\bibitem {AIK}Agrawal, P.N., Ispir, N., Kajla, A., Approximation properties of
Lupas--Kantorovich operators based on Polya distribution, Rend. Circ. Mat.
Palermo, 65, 185--208, (2016).

\bibitem {Agrawal2}Agrawal, P.N., Ispir, N., Kajla, A., Approximation
properites of B\'{e}zier-summation integral type operators based on
P\'{o}lya-Bernstein functions, Appl. Math. Comput. 259, 533--539, (2015).

\bibitem {Barbosu}Barbosu, D., Kantorovich Stancu type operators, J.
Inequal.Pure Appl. Math., 5(3), Article ID 53, (2004).

\bibitem {Bernstein}Bernstein, S.N., Demonstration du theoreme de Weierstrass
Fondee sur le calcul des probabilites, Comp. Comm. Soc. Mat. Charkow Ser. 13,
no. 2, 1--2, (1912).

\bibitem {Cardenas}C\'{a}rdenas-Morales, D., Gupta, V., Two families of
Bernstein--Durrmeyer type operators, Applied Mathematics and Computation, 248,
342-353, (2014).

\bibitem {Cetin}\c{C}etin, N., Ba\c{s}canbaz-Tunca, G., Approximation by a new
complex generalized Bernstein operators, An. Univ. Oradea Fasc. Mat, 26(2),
127--139, (2019).

\bibitem {Deo}Deo, N., Dhamija, M., Micl\u{a}u\c{s}, D., Stancu--Kantorovich
operators based on inverse P\'{o}lya--Eggenberger distribution, Applied
Mathematics and Computation, 273, 281-289, (2016).

\bibitem {Diaz}Diaz, R., Pariguan, E., On hypergeometric functions and
Pochhammer $k$-symbol, Divulgaciones Matem\'{a}ticas, Vol. 15(2), 179-192, (2007).

\bibitem {GG}Gadjiev, A.D., Ghorbanalizadeh, A.M., Approximation properties of
a new type Bernstein--Stancu polynomials of one and two variables, Appl. Math.
Comput. 216(3), 890--901, (2010).

\bibitem {Gupta}Gupta, V., Rassias, T., Lupa\c{s}--Durrmeyer operators based
on P\'{o}lya distribution, Banach J. Math. Anal. 8(2), 146--155, (2014).

\bibitem {Kajla}Kajla, A., Micl\u{a}u\c{s}, D., Some smoothness properties of
the Lupa\c{s}-Kantorovich type operators based on P\'{o}lya distribution,
Filomat, 32(11), 3867-3880, (2018).

\bibitem {Li}Li, S.F., Dong, Y., $k$-Hypergeometric series solutions to one
type of non-homogeneous k-hypergeometric equations., Symmetry, 11, 262, (2019).

\bibitem {Kokolo}Kokologiannaki, C. G., Properties and inequalities of
generalized $k$-gamma, beta and zeta functions, Int. J. Contemp. Math.
Sciences, Vol. 5(14), 653-660, (2010).

\bibitem {Krasniqi}Krasniqi V., A limit for the $k$-Gamma and $k$-Beta
Function, Int. Math.Forum, 5. N33., (2010).

\bibitem {Lupas}Lupas, L., Lupas, A., Polynomials of binomial type and
approximation operators. Stud. Univ. Babes-Bolyai Math. 32(4), 61--69 (1987)

\bibitem {Miclaus}Miclaus, D., The revision of some results for Bernstein
Stancu type operators, Carpathian J. Math, 28(2), 289--300, (2012).

\bibitem {MO}Mohiuddine, S.A., \"{O}zger, F., Approximation of functions by
Stancu variant of Bernstein Kantorovich operators based on shape parameter
$\alpha$, RACSAM, 114:70, (2020).

\bibitem {Mubeen}Mubeen, S., Rehman, A., A note on $k$-gamma function and
pochhammer $k$-symbol, Journal of Informatics and Mathematical Sciences 6(2),
93-107, (2014).

\bibitem {Mubeen1}Mubeen, S., $k$-Analogue of Kummer's first formula, J.
Inequal. Spec. Funct., 3(3), 41--44, (2012).

\bibitem {Neer}Neer, T., Agrawal, P. N., A genuine family of
Bernstein-Durrmeyer type operators based on P\'{o}lya basis functions,
Filomat, 31(9), 2611-2623, (2017).

\bibitem {Nowak}Nowak, G., Approximation properties for generalized
q-Bernstein polynomials. J. Math. Anal. Appl. 350, 50--55, (2009)

\bibitem {Opris}Opris, A.A., Approximation by modified Kantorovich Stancu
operators, Journal of Inequalities and Applications, 2018:346, (2018).

\bibitem {Ost}Ostrovska, S., Turan, M., The distance between two limit
$q$-Bernstein operators, Rocky Mountain Journal of Mathematics, 50(3),
1085-1096, (2020).

\bibitem {Ozars}\"{O}zarslan, M.A., Duman, O., Smoothness properties of
modified Bernstein-Kantorovich operators, Numer. Funct. Anal. Optim. 37,
92--105, (2016).

\bibitem {RMK}Rahman, S., Mursaleen, M., Khan, A., A Kantorovich variant of
Lupas--Stancu operators based on P\'{o}lya distribution with error estimation,
RACSAM, 114:75, (2020).

\bibitem {Razi-Tez}Razi, Q., Approximation of functions by Bernstein type
operators, Master Thesis, Aligarh Muslim University, Aligarh, India, 1983.

\bibitem {Razi}Razi, Q., Approximation of a function by Kantorovich type
operators, Mat. Vesnik, 3, 183-192, (1989).

\bibitem {Stancu}Stancu, D. D., Approximation of functions by a new class of
linear polynomial operators, Rev. Roumaine Math. Pures Appl, 13(8), 1173-1194, (1968).

\bibitem {Volko}Volkov, V.I., On the convergence of sequences of linear
positive operators in the space of two variables, Dokl. Akad. Nauk. SSSR
(N.S.), 115, 17--19, (1957).

\end{thebibliography}

\end{document}